%% file: disc_var_scheme.tex
\documentclass[11pt]{article}
\usepackage{amsmath, amsfonts, amsthm, amssymb}
\usepackage{graphics}

\usepackage{fullpage}
\usepackage{dsfont}
\usepackage{color}
\usepackage{epsfig,psfrag}

\newcommand{\INTT}{\int_{\torus^3}}
\newcommand{\PP}{\mathbb{P}}
\newcommand{\what}{\widehat}

\title{ Stability of fully discrete variational schemes for elastodynamics with a polyconvex stored energy}

\author{Alexey Miroshnikov
\thanks {Department of Mathematics, University of California, Los Angeles, amiroshn@math.ucla.edu}\;
}

\date{}

\numberwithin{equation}{section}

\begin{document}


\maketitle

\abstract{In this article we develop a fully discrete variational  scheme that approximates the equations of three dimensional elastodynamics with polyconvex stored energy. The fully discrete scheme is based on a time-discrete variational scheme developed by S.~Demoulini, D.~M.~A.~Stuart and A.~E.~Tzavaras (2001). We show that the fully discrete scheme is unconditionally stable. The proof of stability is based on a relative entropy estimation for the fully discrete approximates.
}


\input my_macros

%
%

\section{Introduction}

The equations describing the evolution of a continuous medium with nonlinear
elastic response and zero body forces in referential description are given by
\begin{equation}\label{ELASTINTRO1}
\frac{\partial^2 y}{\partial t^2}=\mathrm{div}\, S(\nabla y)
\end{equation}
where $y(x,t)  : \Omega \times [0,\infty) \to{\RR}^3$ stands for
the elastic motion, $S$ for the Piola-Kirchhoff stress tensor and
the region $\Omega$ is the reference configuration of the elastic
body.

\par

The equations
\eqref{ELASTINTRO1} are often recast as a system of conservation
laws,
\begin{equation}\label{ELASTINTRO2}
\qquad\qquad\qquad\begin{aligned}
\partial_{t} v_i &= \partial_{x_{\alpha}}  S_{i\alpha}(F)\\
\partial_{t}F_{i\alpha}&=\partial_{x_{\alpha}}v_i,
\end{aligned} \qquad \alpha=1,\dots,3
\end{equation}
for the velocity $v=\partial_t \SP y \in \RR^3$  and the
deformation gradient $F=\nabla y \in \mdd{3}$ and we use the
summation convention over repeated indices. The differential
constraints
\begin{equation}\label{GRADCONSTR}
\partial_{\beta} F_{i\alpha} - \partial_{\alpha} F_{i\beta} = 0
\end{equation}
are propagated from  the kinematic equation
\eqref{ELASTINTRO2}$_2$ and are an involution \cite{Dafermos86}.

\par

For  {\it hyperelastic} materials
the first Piola-Kirchhoff stress tensor $S(F)=\frac{\del {W}}{\del F}(F)$ is expressed as the
gradient of the stored-energy function of the elastic body
\begin{equation*}
{W}(F):\mddplus{3} \to \RR^3 \quad \mbox{where} \quad \mddplus{3}:=\bigl\{F\in\mdd{3}: \,\det{F}>0 \bigr\}.
\end{equation*}

\par

Convexity of the stored energy is, in general,  incompatible with certain physical requirements and  is not a natural assumption. As an alternative,  we consider {\em polyconvex} stored energy ${W}$, which means that
\begin{equation}\label{POLYCONVEXITY}
{W}(F)=G(F, \cof{F}, \det{F})=G \circ{\Phi}(F), \quad {\Phi}(F)=(F, \cof{F}, \det{F})
\end{equation}
where
\begin{equation}\label{GFUNCDEF}
G=G(\xi)=G(F,Z,w):\mddplus{3}\times\mdd{3}\times\RR \,\cong\,
\RR^{19} \to \RR
\end{equation}
is a convex function.

\par

For polyconvex stored energies \eqref{POLYCONVEXITY} the system of
elasticity \eqref{ELASTINTRO1} is expressed by
\begin{equation}\label{PXELASTSYS}
\begin{aligned}
\partial_{t} v_i&=\partial_{x_{\alpha}}\biggl(\pd {G}{\xi_{A}}({\Phi}(F))
\pd{{\Phi}^{A}}{F_{i\alpha}}(F)\biggr)\\
\partial_{t} F_{i\alpha}&=\partial_{x_{\alpha}} v_i
\end{aligned}
\end{equation}
which is equivalent to \eqref{ELASTINTRO1} subject to differential
constrains \eqref{GRADCONSTR} that are an involution \cite{Dafermos86}: if they are satisfied
for $t=0$ then \eqref{PXELASTSYS} propagates \eqref{GRADCONSTR} to
satisfy for all times. Thus the system \eqref{PXELASTSYS} is
equivalent to systems \eqref{ELASTINTRO1} whenever $F(\cdot,0)$ is
a gradient. Most importantly the system \eqref{PXELASTSYS} is endowed with the entropy identity
\begin{equation}\label{PXENID}
\partial_t\biggl(\frac{|v|^2}{2}+G({\Phi}(F))\biggr)-\partial_{x_{\alpha}}\biggl(v_i\,\pd{G}{
\xi_A}({\Phi}(F))\,\pd{{\Phi}^A}{F_{i\alpha}}(F)\biggr)=0.
\end{equation}

In the present work we are concerned with the design of a numerical method along with a numerical analysis theory for solutions to the equations of elastodynamics with polyconvex stored energies.  The main objective of this work is to develop fully-discrete numerical scheme for equation \eqref{PXELASTSYS} based on the time-discrete variational method introduced by Stuart, Demoulini and Tzavaras \cite{DST2}. The variational method in \cite{DST2} is used to approximate solutions of elasticity equations with polyconvex stored energy. In \cite{DST2} the equations \eqref{PXELASTSYS} are embedded into a larger system 
\begin{equation}\label{EXTELASTSYS}
\begin{aligned}
\partial_{t} v_i&=\partial_{x_{\alpha}}\biggl(\pd {G}{\xi_{A}}(\xi)
\pd{{\Phi}^{A}}{F_{i\alpha}}(F)\biggr)\\
\partial_{t} \xi_A &= \partial_{\alpha}\Big(\pd{\Phi^{A}}{F_{i\alpha}}(F)\,v\Big)
\end{aligned}
\end{equation}
that has variables $v\in\RR^3$, $\Xi \in \RR^{19}$ and is equipped with a convex entropy $\eta(v,\Xi)=\frac{1}{2}v^2 +G(\Xi)$. The convexity of the entropy $\eta$ allows the authors to employ variational techniques in time-discrete settings.  The variational method produces the sequence of spatial iterates $\{v^n,\Xi^n\}_{n \geq 1}$ that solve the time-discretized version of the enlarged elasticity system,
\begin{equation}\label{DISCSYS}
\begin{aligned}
\tfrac{1}{\Delta t}{(v^n_i-v^{n-1}_i)}&=\partial_{\alpha}\Big(\pd{G}{\varXi_{A}}(\varXi^n)\,\,\pd{\Phi^{A}}{F_{i\alpha}}(F^{n-1})\Big)\\ \quad
\tfrac{1}{\Delta t}{(\varXi^n_A-\varXi^{n-1}_A)}&=\partial_{\alpha}\Big(\pd{\Phi^{A}}{F_{i\alpha}}(F^{n-1})\,v^n_i\Big)
\end{aligned} \quad \mbox{in} \quad \cD'(\torus^3)\,,
\end{equation}
and give rise to time-continuous approximates converging to the solution of elastodynamics before shock formation (see \cite{MT12}).

\par

The main challenge in the present work is to adapt the minimization framework of \cite{DST2} to space-discrete settings, which can be accomplished by an appropriately designed finite element method. To realize the finite element scheme it is essential a) to identify appropriate finite element spaces used in space discretization,  b) to provide an error estimate for the approximation, and c) to test the finite element scheme numerically.

\par

In our present work we  introduce the following numerical method: given appropriate finite element spaces $U_h, H_h$, and data
$(v^{n-1}_h,\Xi^{n-1}_h) \in U_h \times H_h$ at time step $t=t_{n-1}$
construct the next iterate by solving
\begin{equation}\label{FESCHEME}
\begin{aligned}
\tfrac{1}{\Delta t}\big( v^n_h - v^{n-1}_h , \varphi_h \big)  &\,= \,-
\big( D_{\Xi}G(\Xi^n_h), D_{F}\Phi(F^{n-1}_h) \nabla \varphi_h \big)\,, \quad \forall\varphi_h \in U_h\\[3pt]
\mbox{where} \quad \Xi^n_h &\,=\, \Xi^{n-1}_h + \Delta t \SP \big( D_{F}\Phi(F^{n-1}_h) \nabla v^n \big) \in H_h\,.
\end{aligned}
\end{equation}
In this work we introduce suitable finite element spaces that render the finite element scheme \eqref{FESCHEME} {unconditionally stable}, which is a necessary requirement for any reliable numerical method. The spaces of test functions are rich enough that an important (gradient) conservation property $\frac{1}{\Delta t}{F^j_h-F^{j-1}_h}=\nabla v^j_h$ is satisfied by the finite element approximation at each time step. This property is essential in adapting the method of \cite{DST2} to a fully discrete scheme.
The existence of numerical solutions to \eqref{FESCHEME} is obtained using minimization principles. 

\par
In this article we establish the stability of numerical solutions and derive the relative entropy identity which is central to establishing the convergence and providing an error estimate. Our stability analysis follows in spirit the work of Miroshnikov and Tzavaras \cite{MT12} where the authors established the direct convergence of iterates produced by the time-discrete scheme \eqref{DISCSYS}. Specifically, following \cite{MT12}, we consider the relative entropy $\eta^r=\eta^r(x,t)$
 \begin{equation*}
\eta^r=\tfrac{1}{2}|V^{(\Delta t,h)}-\bar{V}|^2+
\big[G(\Xi^{(\Delta t,h)})-G(\bar{\Xi})-{D_{\Xi}G(\bar{\Xi})}(\Xi^{(\Delta t,h)}-\bar{\Xi})\big]
\end{equation*}
that estimates the difference between time-continuous approximations $V^{(\Delta t,h)},\Xi^{(\Delta t,h)}$
generated by the numerical scheme and the classical solution $(\bar{V},\bar{\Xi})$
of the extended elasticity system and derive {\em the relative entropy identity}
\begin{equation}\label{RENTIDINTRO}
\int_{\Omega} \SP \Big\{ \partial_t \SP \eta^r(x,t) + \del_{x_\alpha} q^r_{\alpha}(x,t) \Big\} \, dx = \int_{\Omega} \Bigl( -\frac{1}{\Delta t} D + Q+E\Bigr) \SP dx
\end{equation}
that monitors the time evolution of $\eta^r$. Here $D>0$ is the dissipation produced by the scheme, $Q$ is the term equivalent to $\eta^r$, and $E$ is the error term. The relative entropy identity is central to establishing the stability of the scheme and providing an error estimate; establishing the convergence of the approximates is the subject of future investigations.


\par

\section{Time-discrete variational approximation scheme }

\subsection{Time-discrete scheme and its stability}
In this section we briefly describe the semi-discrete variational scheme of Demoulini, Stuart and Tzavaras \cite{DST2} as well as the
result of the article \cite{MT12} in which to avoid inessential difficulties the authors work with periodic boundary conditions, {\it
i.e.}, the spatial domain $\Omega=\torus^3$ is taken to be the three-dimensional torus.

\par\medskip

The work \cite{DST2} uses extensively the properties of so-called null-Lagrangians. To this end we recall its definition:
\begin{definition} A continuous
function $L(F):\mdd{3} \to \RR$ is a {\it null-Lagrangian} if
\begin{equation}\label{NLDEF}
    \int_{\Omega} \, L(\nabla (u + \varphi )(x))\, dx = \int_{\Omega} \, L(\nabla u(x))\, dx
\end{equation}
for every bounded open set $\Omega \subset \RR^3$ and for all $u \in C^1(\bar{\Omega};\RR^3)$, $\varphi \in
C^{\infty}_0(\Omega;\RR^3)$.
\end{definition}

\par

It turns out that the components of ${\Phi}(F)$ defined in \eqref{POLYCONVEXITY} are null-Lagrangians and satisfy
\begin{equation}\label{SNLP}
\partial_{\alpha}\biggl(\pd{{\Phi}^A}{F_{i\alpha}}(\nabla{u})\biggr) =
0, \qquad A=1,\dots,19
\end{equation}
for any smooth $u(x):\RR^3 \to\RR^3$. Therefore, for smooth solutions $(v,F)$ of \eqref{PXELASTSYS}, the null-Lagrangians
${\Phi}^A(F)$ satisfy the transport identities \cite{DST}
\begin{equation}\label{NLTID}
\partial_t {\Phi}^A(F) = \partial_{\alpha} \biggl(\pd{{\Phi}^A}{F_{i\alpha}}(F)v_i
\biggr), \quad \forall F \;\; \mbox{with} \;\;
\partial_{\beta}F_{i\alpha}=\partial_{\alpha}F_{i\beta}.
\end{equation}
Due to the identities \eqref{NLTID} the system of polyconvex elastodynamics \eqref{PXELASTSYS} can be embedded into the
enlarged system \cite{DST}
\begin{equation}\label{EXTSYS}
\begin{aligned}
\partial_{t}v_i&=\partial_{\alpha}\biggl(\pd {G}{\xi_A}(\xi)
\,\,\pd{{\Phi}^{A}}{F_{i\alpha}}(F)\biggr)\\
\partial_{t}{\Phi}_{A}&=\partial_{\alpha}\biggl(\pd{{\Phi}^{A}}{F_{i
\alpha}}(F)\,v_{i}\biggr). \\
\end{aligned}
\end{equation}
The extension has the following properties:
\begin{itemize}

\item[(E\,1)] If $F(\cdot,0)$ is a gradient then $F(\cdot,t)$ remains a gradient $\forall t$.

\item[(E\,2)] If $F(\cdot,0)$ is a gradient and $\xi(\cdot,0)=\xi(F(\cdot,0))$, then $F(\cdot,t)$ remains a gradient and
    $\xi(\cdot,t)=\xi(F(\cdot,t))$, $\forall t$. In other words, the system of polyconvex elastodynamics can be viewed as a
    constrained evolution of \eqref{EXTSYS}.

\item[(E\,3)] The enlarged system admits a convex entropy
\begin{equation}\label{ENTDEF}
\eta(v,\xi) = \tfrac{1}{2} |v|^2 +G(\xi), \quad (v,\xi) \in
\RR^{22}
\end{equation}
and thus is symmetrizable (along the solutions that are gradients).
\end{itemize}

\par\smallskip

Based on the time-discretization of the enlarged system \eqref{EXTSYS} S. Demoulini, D.~M.~A. Stuart and A.~E. Tzavaras
\cite{DST} developed a variational approximation scheme which, for the given initial data
\begin{equation}\label{INITITERCLASS}
\Theta^0:=(v^0, \xi^0)=(v^0,F^0,Z^0,w^0) \in L^2 \times L^p \times
L^2 \times L^2
\end{equation}
and fixed time step $\tau >0$, constructs the sequence of successive iterates
\begin{equation}\label{GENITERCLASS}
\Theta^n:=(v^n, \xi^n)=(v^n,F^n,Z^n,w^n)  \in  L^2 \times L^p
\times L^2 \times L^2, \;\; n \geqslant 1
\end{equation}
with the following properties (see \cite[Lemma 1, Corollary 2]{DST2}):

\begin{itemize}

\item[(P\,1)] The iterate $(v^n,\xi^n)$ is the unique minimizer of the functional
\begin{equation*}
\mathcal{J}[v,\xi] = \int_{\mathbb{T}^3}
\biggl(\tfrac{1}{2}|v-v^{n-1}|^2 + G(\xi)\biggr) dx
\end{equation*}
over the weakly closed affine subspace
\begin{equation*}
\begin{aligned}
\mathcal{C} = \biggl\{ &(v,\xi)\in L^2 \times L^p \times L^2
\times
L^2: \; \mbox{such that} \; \forall \varphi \in C^{\infty}(\mathbb{T}^3)\\
&\quad \int_{\mathbb{T}^3} \biggl(\frac{\xi_{A} -
\xi^{n-1}_{A}}{\tau}\biggr) \varphi \,dx = - \int_{\mathbb{T}^3}
\biggl(\pd{{\Phi}^{A}}{F_{i\alpha}}(F^{n-1}) v_i\biggr) \partial_{\alpha} \varphi \, dx \biggr\}.\\
\end{aligned}
\end{equation*}

\item[(P\,2)] For each $n \geqslant 1$ the iterates satisfy the corresponding Euler-Lagrange equations
\begin{equation}\label{DISCEXTSYS}
\begin{aligned}
\frac{v^n_i-v^{n-1}_i}{\tau}&=\partial_{\alpha}\biggl(\pd{G}{\xi_{A}}(\xi^n)\,\,\pd{{\Phi}^{A}}{F_{i\alpha}}(F^{n-1})\biggr)\\
\frac{\xi^n_A-\xi^{n-1}_A}{\tau}&=\partial_{\alpha}\biggl(\pd{{\Phi}^{A}}{F_{i\alpha}}(F^{n-1})\,v^n_i\biggr)
\end{aligned} \qquad \mbox{in} \quad \cD'(\mathbb{T}^3).
\end{equation}

\item[(P\,3)] If $F^0$ is a gradient, then so is $F^n\,$, $\forall n \geqslant 1$.

\item[(P\,4)] Iterates $v^n$, $n\geqslant 1$ have higher regularity: $v^n \in W^{1,p}(\mathbb{T}^3)$, $\forall n\geqslant 1$.

\item[(P\,5)]There exists $E_0>0$ determined by the initial data such that
\begin{equation}\label{ITERBOUND}
\begin{aligned}
\sup_{n \geqslant \, 0} \Bigl( \| v^n\|^2_{L^2_{dx}}
+\int_{\mathbb{T}^3} G(\xi^n) \,dx \Bigr) + \sum_{n=1}^{\infty}
\|\Theta^n - \Theta^{n-1}\|_{L^2_{dx}}^2  \leqslant E_0.
\end{aligned}
\end{equation}

\end{itemize}

\subsection{Convergence of the time-discrete scheme}

In \cite{MT12} we established the direct convergence of time-continuous interpolates,
\begin{equation}\label{CONTINTPSM}
\begin{aligned}
\what{v}^{(\tau)}(t)&=\sum^{\infty}_{n=1}\Chi^n(t) \Bigl(v^{n-1}+\frac{t-\tau(n-1)}{\tau}(v^n-v^{n-1})\Bigr)\\
\what{\xi}^{(\tau)}(t)&=\bigl(\what{F}^{(\tau)},\what{Z}^{(\tau)},\what{w}^{(\tau)}\bigr)(t)\\
&=\sum^{\infty}_{n=1}\Chi^n(t)\Bigl(
\xi^{n-1}+\frac{t-\tau (n-1)}{\tau}(\xi^n - \xi^{n-1})\Bigr)\\
& \qquad \qquad \mbox{where} \quad \Chi^n(t) = \mathds{1}_{[(n-1)\tau,n\tau)}\,,
\end{aligned}
\end{equation}
constructed in the time-discrete scheme \eqref{DISCEXTSYS} to the solution of elastodynamics before shock formation and provided
the error estimate. The proof is based on the relative entropy method \cite{Dafermos10, DiPerna79} and provides an error estimate
for the approximation before the formation of shocks. This work is the first step towards numerical method used for practical
purposes (eg. computing solutions).

 \par

 To establish convergence we employed the relative entropy argument (see \cite{Dafermos10, DiPerna79}). We considered the relative entropy,
\begin{equation*}
    \eta^r
\SP = \SP  \tfrac{1}{2}|\what{v}^{(\tau)}-{v}|^2+
\big[G(\what{\xi}^{(\tau)})-G({\xi})-\D_{\xi}G({\xi})(\what{\xi}^{(\tau)}-{\xi}) \big]\,,
\end{equation*}
which estimates the difference between time-continuous interpolates $\big(\what{v}^{(\tau)},\what{\xi}^{(\tau)}\big)$ produced
by the scheme and a classical solution $\big({v},{\xi}\big)$ of the enlarged system, and derived the energy identity monitoring the
time evolution of $\eta^r$. I showed that (under appropriate assumptions for growth of $G$) the relative entropy $\eta^r$ satisfies
the identity
\begin{equation}\label{RENTIDSD}
{\partial_t \hspace{1pt} \eta^r-\mathrm{div}\,q^r
=Q-\tfrac{1}{\tau} D+S \quad \mbox{in}
\quad \cD'}
\end{equation}
which monitors its time evolution. Here $D>0$ is the dissipation generated by the scheme, $Q$ is the term equivalent to $\eta^r$,
and $E$ is the time discretization error.
The analysis of the identity 
yielded the main result: if $({v},{F})$ are smooth solutions of the elasticity equations \eqref{ELASTINTRO1} then
\begin{equation*}
 \sup_{t\in[0,T]}  \BBR{ \|\what{v}^{(\tau)} - {v} \|^2_{L^2(\mathbb{T}^3)} + \|\what{\xi}^{(\tau)} - {\Phi}({F}) \|^2_{L^2(\mathbb{T}^3)}  +
\|\what{F}^{(\tau)}-{F}\|^p_{L^p(\mathbb{T}^3)} } = O(\tau).
\end{equation*}

\section{Fully-discrete variational approximation scheme}

\subsection{Stored energy assumptions}

We consider polyconvex stored energy ${W}:\mddplus{3} \to \RR$
\begin{equation}\vspace{-4pt}\label{POLYCONVEXITY2}
{W}(F)=G\circ{\Phi}(F)\,, \quad {\Phi}(F) := (F, \cof{F}, \det{F})
\end{equation}
with
\begin{equation*}
G=G(\xi)=G(F,Z,w):\mdd{3}\times\mdd{3}\times\RR \,\cong\, \RR^{19}
\to \RR \quad \mbox{uniformly convex.}
\end{equation*}

\par
We work with periodic boundary conditions, that is, the spatial domain $\Omega$ is taken to be {the} three dimensional torus
$\mathbb{T}^3$. The indices $i,\alpha, \dots$ generally run over $1,\dots,3$ while $A,B,\dots$ run over $1,\dots,19$. We use the
notation $L^p=L^p(\mathbb{T}^3)$ and $W^{1,p}=W^{1,p}(\mathbb{T}^3)$. Finally, we impose the following convexity and
growth assumptions on $G$:
\begin{itemize}

\item[(H1)] $G\in C^3(\mdd{3}\times \mdd{3} \times \RR; [0,\infty))$ is of the form
\begin{equation}\label{GDECOMP}
G(\xi)=H(F) + R(\xi)
\end{equation}
with $H\in C^3(\mdd{3}; [0,\infty))$ and $R\in C^3(\mdd{3}\times \mdd{3} \times \RR; [0,\infty))$ strictly convex satisfying
\begin{equation*}
\kappa |F|^{p-2}|z|^{2} \, \leqslant \, z^{T}\nabla^2H(F)z \,
\leqslant \, \kappa' |F|^{p-2}|z|^{2}, \,\,\, \forall z \in \RR^9
\end{equation*}
and $\gamma I  \leqslant  \nabla^2R \leqslant \gamma' I\,$ for some fixed $\gamma,\gamma',\kappa,\kappa'>0$ and $p\in
(6,\infty)$.

\item[(H2)] $G(\xi)\,\geqslant\,c_1|F|^p+c_2|Z|^2+c_3|w|^2-c_4$.

\item[(H3)] $G(\xi)\,\leqslant\,c_5(|F|^p+|Z|^2+|w|^2+1)$.

\item[(H4)] $|G_{F}|^{\frac{p}{p-1}}+ |G_Z|^{\frac{p}{p-2}}+|G_w|^{\frac{p}{p-3}}
    \,\leqslant\,c_6\BBR{|F|^p+|Z|^2+|w|^2+1}.$

\item[(H5)]
$\BBA{\frac{\partial^3H}{\partial{F_{i\alpha}}\partial{F_{ml}}\partial{F_{rs}}}} \leqslant c_7 |F|^{p-3}$ and
$\BBA{\frac{\partial^3R}{\partial{{\Phi}_A}
\partial{{\Phi}_B} \partial{{\Phi}_D}}} \leqslant c_8$.

\end{itemize}\par\medskip

\noindent {\bf Notations.} To simplify notation we write
\begin{equation*}
\begin{aligned}
G_{,A}\,(\xi)&=\pd{G}{\xi_{A}}(\xi),& \quad
R_{,A}\,(\xi)&=\pd{R}{\xi_{A}}(\xi),&\\
H_{,i\alpha}\,(F)&=\pd{H}{F_{i\alpha}}(F),& \quad
{\Phi}^{A}_{,i\alpha}\,(F)&=\pd{{\Phi}^{A}}{F_{i\alpha}}(F).&
\end{aligned}
\end{equation*}
For each $i,\alpha =1,\dots,3$ we set
\begin{equation}\label{GCOMPFLD}
\begin{aligned}
g_{i\alpha}(\xi,\tilde{F})=\pd{G}{{\Phi}_{A}}\,(\xi)
\,\pd{{\Phi}^{A}}{F_{i\alpha}}\,(\tilde{F})=G_{,A}(\xi) \SP
{\Phi}^A_{,i\alpha}(\tilde{F})
\end{aligned}
\end{equation}
for $\xi=(F,Z,w) \in \RR^{19}\,, \tilde{F} \in \RR^9$ (where we use {the} summation convention over repeated indices)
and set the corresponding fields by
\begin{equation}\label{GFLD}
g_i: \RR^{19}\times\RR^9\to\RR^3,  \quad
g_i(\xi,\tilde{F}):=(g_{i1},g_{i2},g_{i3})(\xi,\tilde{F}).
\end{equation}

{
\begin{remark}\rm
The hypothesis (H1) can be replaced with a more general one
\begin{equation}\label{ff}\tag{H1$'$}
G(F)=H_1(F)+H_2(Z)+H_3(w)+R(\xi)
\end{equation}
which leads to a delicate error estimation analysis.
\end{remark} }

\subsection{Motivation for the scheme}

In this section we consider fully-discrete scheme induced by the equations \eqref{DISCEXTSYS}. We first prove an elementary
lemma that highlights some of the properties of null-Lagrangians ${\Phi}(F)$:
\begin{lemma}[\textbf{null-Lagrangian properties}]\label{DIVNLLMM}
Let  $q>2$ and $r \geqslant \tfrac{q}{q-2}$. Then, if
\begin{equation*}
  u \in
W^{1,q}(\mathbb{T}^3;\RR^3)\,, \quad z\in W^{1,r}(\mathbb{T}^3)
\end{equation*}
we have
\begin{equation}\label{NLPROPEXT}
\begin{aligned}
\partial_{\alpha}\biggl(\frac{\del {\Phi}^A}{\del F_{i\alpha}}\BBR{\nabla{u}}\biggr)&=0\\
\partial_{\alpha}\biggl(\frac{\del{\Phi}^A}{\del F_{i\alpha}}(\nabla{u})z\biggr) &= \frac{\del{\Phi}^A}{\del F_{i\alpha}}(\nabla{u})\,\partial_{\alpha}z
\end{aligned} \quad \mbox{in} \quad \cD'(\mathbb{T}^3)\\
\end{equation}
for each $i=1,\dots,3$ and $A=1,\dots,19$.
\end{lemma}

\begin{proof}
Observe that
\begin{equation*}
\Phi_{,i \alpha}(\nabla u) \leqslant 1 +|\nabla u|+|\nabla u|^2
\quad  \Rightarrow \quad \PSD(\gr{u})\in L^{q/2}(\torus^3).
\end{equation*}
Hence by \eqref{SNLP} and the density argument we get \eqref{NLPROPEXT}$_1$. Next, notice that
\begin{equation*}
\PSD(\gr{u})z, \;\; \PSD(\gr{u})\,\partial_{\alpha}z \in
L^1(\torus^3).
\end{equation*}
Then taking arbitrary $\varphi \in C^{\infty}(\torus^3)$ we obtain
\begin{equation*}\vspace{3pt}
\begin{aligned}
\int_{\torus^3} \Bigl(\PSD&
(\gr{u})\,z\Bigr)\partial_{\alpha}\varphi \,dx \\
&=\int_{\torus^3}
\Bigl(\PSD(\gr{u})\Bigr)\partial_{\alpha}(z\,\varphi)\,dx -
\int_{\torus^3}\Bigl(\PSD(\gr{u})\,\partial_{\alpha}z\Bigr)\varphi\,dx = I_1-I_2.\\
\end{aligned}
\end{equation*}
Since $z\varphi\in W_0^{1,r} \bigcap W^{1,q^{*}}$, the property \eqref{NLPROPEXT}$_1$ and the density argument imply
$I_1=0$ and hence
\begin{equation*}
\int_{\torus^3} \Bigl(\PSD
(\gr{u})\,z\Bigr)\partial_{\alpha}\varphi\,dx = -I_2 =
\int_{\torus^3}\Bigl(\PSD(\gr{u})\,\partial_{\alpha}z\Bigr)\varphi\,dx.
\end{equation*}
\end{proof}

Using the above lemma and the properties (P3) and (P4), we conclude that the spatial iterates $v^n,\xi^n$ constructed in
\eqref{DISCEXTSYS} solve the system
\begin{equation}\label{DISCEXTSYS2}
\begin{aligned}
\frac{v^n_i-v^{n-1}_i}{\tau}&=\partial_{\alpha}\biggl(\pd{G}{\xi_{A}}(\xi^n)\,\,\pd{{\Phi}^{A}}{F_{i\alpha}}(F^{n-1})\biggr)\\
\frac{\xi^n_A-\xi^{n-1}_A}{\tau}&=\frac{\partial{\Phi}^A}{\partial F_{i\alpha}}(F^{n-1})
\SP \del_{x_{\alpha}}{v_i^n}
\end{aligned} \qquad \mbox{in} \quad \cD'(\mathbb{T}^3).
\end{equation}
which in the shorter form can be expressed
\begin{equation}\label{DISCEXTSYS2S}
\begin{aligned}
\biggl( \frac{v^n - v^{n-1}}{\tau} , \varphi \biggr) & = - \bigl(
DG(\xi^n), D{\Phi}(F^{n-1}) \nabla \varphi \bigr), &\varphi \in
C^{\infty}\bigl( \torus^3 ; \SP \RR^3
\bigr)& \\[3pt]
\biggl( \frac{\xi^n - \xi^{n-1}}{\tau} , \psi \biggr) &=
\bigl(D{\Phi}(F^{n-1}) \nabla v^n, \psi \bigr),  &\psi \in
C^{\infty}\bigl( \torus^3 ; \SP \RR^{19} \bigr)&.
\end{aligned}
\end{equation}

\par\smallskip

\begin{remark}\rm
The system \eqref{DISCEXTSYS2} is equivalent to \eqref{DISCEXTSYS} for smooth solutions or functions satisfying (P3)-(P4), but in
a distributional sense they are not equivalent. Observe that the product of a function and (possibly) a measure on the right-hand
side of the second equation \eqref{DISCEXTSYS2} may not be defined unless we require $v$ to have a better regularity.
\end{remark}

\subsection{Fully-discrete scheme and stability}

Based on the previous discussion let us investigate a possibility for a fully-discrete scheme based on \eqref{DISCEXTSYS2S}. As
before, let $\tau>0$ be fixed time-step and $h>0$ correspond to a space-step. Set spaces
\begin{equation}
\begin{aligned}
U_h &= \Bigl\{ \varphi_h \in C\bigl(\torus^3; \SP \RR^3\bigr): \:
\varphi_h|_K \in [\mathcal{P}_k(K)]^3,\: K\in\mathcal{T}_h(\torus^3) \Bigr\}\\[6pt]
H^F_h &= \Bigl\{ A_h \in L^2(\torus^3; \SP \mdd{3}): \:
A_h|_K \in [\mathcal{P}_{k-1}(K)]^9,\: K\in\mathcal{T}_h(\torus^3) \Bigr\}\\[6pt]
H^Z_h &= \Bigl\{ B_h \in L^2(\torus^3; \SP \mdd{3}): \:
B_h|_K \in [\mathcal{P}_{2(k-1)}(K)]^9,\: K\in\mathcal{T}_h(\torus^3) \Bigr\}\\[6pt]
H^w_h &= \Bigl\{ d_h \in L^2(\torus^3): \: d_h|_K \in
\mathcal{P}_{3(k-1)}(K),\: K\in\mathcal{T}_h(\torus^3) \Bigr\}.
\end{aligned}
\end{equation}
and let $\PP^U$, $\rm \PP^F$, $\rm \PP^Z$ and $\rm \PP^w$ denote the standard orthogonal projectors,
\begin{equation}\label{ORTHPROJ}
\begin{aligned}
    {\rm{\mathbb{P}}^U}&: L^2(\torus^3; \SP \RR^3) \, \to \, U_h\\
    {\rm{\mathbb{P}}^F}&: L^2(\torus^3; \SP \mdd{3}) \, \to \, H^F_h\\
    {\rm{\mathbb{P}}^Z}&: L^2(\torus^3; \SP \mdd{3}) \, \to \, H^Z_h\\
    {\rm{\mathbb{P}}^w}&: L^2(\torus^3 ) \, \to \, H^w_h,
\end{aligned}
\end{equation}
and set the operator ${\rm \mathbb{P}^H} := ({\rm{\mathbb{P}}^F}, {\rm{\mathbb{P}}^Z}, {\rm{\mathbb{P}}^w})$.

\par\medskip

Consider the following {\it fully-discrete scheme}: Given the $(n-1)$-st iterate
\begin{equation*}
(v^{n-1}_h,\xi^{n-1}_h) = (v^{n-1}_h,
F^{n-1}_h,Z^{n-1}_h,w^{n-1}_h) \in U_h \times H^{F}_h \times
H^{Z}_h \times H^{w}_h
\end{equation*}
find the $n$-th iterate
\begin{equation*}
(v^{n}_h,\xi^{n}_h) = (v^{n}_h, F^{n}_h,Z^{n}_h,w^{n}_h) \in U_h
\times H^{F}_h \times H^{Z}_h \times H^{w}_h
\end{equation*}
by solving
\begin{equation}\label{FDISCEXTSYSNONDIV_SHORT}
\begin{aligned}
\biggl( \frac{v^n_h - v^{n-1}_h}{\tau} , \varphi_h \biggr) & = -
\bigl( DG(\xi^n_h), D{\Phi}(F^{n-1}_h) \nabla \varphi_h \bigr),
\\[3pt]
\biggl( \frac{\xi^n_h - \xi^{n-1}_h}{\tau} , \psi_h \biggr) &=
\bigl(D{\Phi}(F^{n-1}) \nabla v^n, \psi_h \bigr)
\end{aligned}
\end{equation}
for all $\varphi_h \in U_h$, $\psi_h \in H^F_h \times H^Z_h \times H^w_h$.

\par\medskip

Assuming that we are able to solve \eqref{FDISCEXTSYSNONDIV_SHORT} we will try to establish {\it a priori} estimates similar to
those in \eqref{ITERBOUND}. By assumption
\begin{equation}\label{CONSTRPROP1}
\xi^n_h - \xi^{n-1}_h \in H^F_h \times H^Z_h \times H^w_h
\end{equation}
and hence
\begin{equation}\label{PROJPROP1}
\Bigl(\xi^n_h - \xi^{n-1}_h, \, {\rm{\mathbb{P}^H}} DG(\xi^n_h)
\Bigr) = \Bigl(\xi^n_h - \xi^{n-1}_h, \, DG(\xi^n_h) \Bigr).
\end{equation}

\par\medskip

Next, recall that ${\Phi}(F) = (F, \cof{F}, \det{F})$ and therefore
\begin{equation*}
 \begin{aligned}
 \frac{\del F_{i\alpha}}{\del F_{j\beta}} = \delta_{ij}
 \delta_{\alpha\beta}, \quad
  \frac{\del \SP (\cof{F})_{i\alpha}}{\del F_{j\beta}}
= \eps_{ijk} \eps_{\alpha\beta\gamma}F_{k\gamma}, \quad \frac{\del
\det{F}}{\del F_{j\beta}}  = (\cof(F))_{j\beta}
 \end{aligned}
\end{equation*}
for each $i,j, \alpha,\beta = 1,\dots 3$. Then, in view of the fact that
\begin{equation*}
F^{n-1}_h \in H^F_h,\, \cof(F^{n-1}_h) \in H^Z_h \quad and \quad
\nabla v^n_h \in H^F_h,
\end{equation*}
we conclude that\vspace{3pt}
\begin{equation*}\vspace{3pt}
\begin{aligned}
&\biggl[\frac{\del ({\Phi}^1,{\Phi}^2,\dots,{\Phi}^9)}{\del
F}(F^{n-1}_h)\biggl] \nabla v^n_h \, = \, \nabla v^n_h  \, \in \,
H^F_h\\[3pt]
&\quad\qquad\qquad\biggl[\frac{\del
({\Phi}^{10},{\Phi}^2,\dots,{\Phi}^{18})}{\del F}(F^{n-1}_h)\biggl] \nabla
v^n_h \,
\in  \, H^Z_h\\[3pt]
&\quad\qquad \qquad\qquad\biggl[\frac{\del {\Phi}^{19}}{\del
F}(F^{n-1}_h)\biggl] \nabla
v^n_h \, = \, \cof{F}:\nabla v^n_h \, \in \, H^w_h.\\[3pt]
\end{aligned}
\end{equation*}
Thus \vspace{3pt}
\begin{equation}\label{CONSTRPROP2}
\begin{aligned}
D{\Phi}(F^{n-1}_h) \nabla v^n_h \, \in \, H^F \times H^Z \times H^w
\end{aligned}
\end{equation}
and therefore
\begin{equation}\label{PROJPROP2}
\Bigl( D{\Phi}(F^{n-1}_h) \nabla v^n_h , {\rm \mathbb{P}^H}
DG(\xi^n_h) \Bigr) \,= \, \Bigl( D{\Phi}(F^{n-1}_h) \nabla v^n_h ,
 DG(\xi^n_h) \Bigr).
\end{equation}

\par\medskip

Next, set $\varphi_h = v^n_h \SP$, $\psi_h = {\rm\mathbb{P}^H} DG(\xi^n)$ and apply it to \eqref{FDISCEXTSYSNONDIV_SHORT}.
It leads to
\begin{equation*}
\begin{aligned}
\Bigl( {v^n_h - v^{n-1}_h} , \nabla v^n_h \Bigr) & = - \tau \Bigl(
DG(\xi^n_h), D{\Phi}(F^{n-1}_h) \nabla v^n_h \Bigr)
\\[3pt]
\Bigl( {\xi^n_h - \xi^{n-1}_h}, {\rm\mathbb{P}^H} DG(\xi^n) \Bigr)
&= \tau \Bigl(D{\Phi}(F^{n-1}) \nabla v^n, {\rm\mathbb{P}^H}
DG(\xi^n) \Bigr)
\end{aligned}
\end{equation*}
which, in view of \eqref{PROJPROP1} -- \eqref{PROJPROP2}, implies
\begin{equation*}
\begin{aligned}
\Bigl( {v^n_h - v^{n-1}_h} , \nabla v^n_h \Bigr) + \tau \Bigl(
DG(\xi^n_h), \xi^n_h-\xi^{n-1}_h \Bigr) = 0.
\end{aligned}
\end{equation*}
The above identity can be rewritten as
\begin{equation*}
\begin{aligned}
\tfrac{1}{2} \BBN{v^n_h}_{L^2(\torus^3)}^2 + \tfrac{1}{2}
\BBN{v^n_h-v^{n-1}_h}_{L^2(\torus^3)}^2 + \bigl(DG(\xi^n_h),
\xi^n_h - \xi^{n-1}_h \bigr) =
\tfrac{1}{2}\BBN{v^{n-1}_h}_{L^2(\torus^3)}^2
\end{aligned}
\end{equation*}
and hence, by the convexity of $G$, we get an {\it a priori} estimate
\begin{equation}\label{STABEST2}
\begin{aligned}
\tfrac{1}{2} \BBN{v^n_h}_{L^2(\torus^3)}^2 + \tfrac{1}{2}
\BBN{v^n_h-v^{n-1}_h}_{L^2(\torus^3)}^2 + &\int_{\torus^3}
\,G(\xi^n_h)\,dx \\[2pt]
\,\leqslant \, \tfrac{1}{2}\BBN{v^{n-1}_h}_{L^2(\torus^3)}^2 +
&\int_{\torus^3} \,G(\xi^{n-1}_h)\,dx .
\end{aligned}
\end{equation}

\par\smallskip

\begin{remark}\rm One actually may avoid introducing orthogonal projector
operator (directly). Notice that \eqref{FDISCEXTSYSNONDIV_SHORT} implies that
\begin{equation*}
\frac{1}{\tau} \bigl(\xi^n_h - \xi^{n-1}_h\bigr) - D{\Phi}(F^{n-1}_h)
\nabla v^n_h \, \in \, \bigl( H^F_h \times H^Z_h \times H^w_h
\bigr)^{\perp}
\end{equation*}
while by \eqref{CONSTRPROP1} and \eqref{CONSTRPROP2} we have
\begin{equation*}
\frac{1}{\tau} \bigl(\xi^n_h - \xi^{n-1}_h\bigr) - D{\Phi}(F^{n-1}_h)
\nabla v^n_h \, \in \, H^F_h \times H^Z_h \times H^w_h.
\end{equation*}
This tells us that
\begin{equation}\label{CONSTRPROP3}
\frac{1}{\tau} \bigl(\xi^n_h - \xi^{n-1}_h\bigr) - D{\Phi}(F^{n-1}_h)
\nabla v^n_h \, = \, 0
\end{equation}
and hence, setting $\varphi_h = v^n_h \DSP$ in \eqref{FDISCEXTSYSNONDIV_SHORT}$_1$ and using the above identity, we get
the same estimate as before. \par\medskip

Finally, notice that the first nine identities in \eqref{CONSTRPROP3} are
\begin{equation}\label{CONSTRPROP4}
\frac{1}{\tau} \bigl(F^n_h - F^{n-1}_h\bigr) - \nabla v^n_h \, =
\, 0
\end{equation}
which suggest that if $F^{n-1}_h$ is a gradient then $F^{n}$ is a gradient, a very useful property if one needs to exploit
null-Lagrangian structure; see, e.g., \cite{MT12}.

\end{remark}

\section{Relative entropy identity}

\subsection{Relative entropy identity in the smooth regime} \label{SENERGYMOTIV}

In this section we derive the relative entropy identity among the two {\it smooth} solutions. This will help to grasp the main idea
behind the calculations and, in addition, explain the need for hypotheses (H1)-(H5) on the stored energy ${W}(F)$; see Section
\eqref{SENERGYMOTIV}. Thus, suppose that
\begin{equation*}
({\what{v}},{\what{\xi}})=({\what{v}},{\what{F}},{\what{Z}},{\what{w}}), \quad ({v},{\xi})=({v},{F},Z,w)
\end{equation*}
are two smooth solutions to the extended system \eqref{EXTSYS} with ${\what{F}}(\cdot,0)$, ${F}(\cdot,0)$ gradients. Define the
relative entropy among the two solutions by
\begin{equation}\label{RENTDEFINTRO}
\eta^r({\what{v}},{\what{\xi}}; \SP {v}, {\xi}):=
\eta({\what{v}},{\what{\xi}})-\eta({v},{\xi})-\D\eta({v},{\xi})({\what{v}}-{v},{\what{\xi}}-{\xi})\\
\end{equation}
with $\eta$ given by \eqref{ENTDEF}. The relative flux in this case will turn out to be
\begin{equation}\label{RFLUXINTRO}
q^r_d({\what{v}},{\what{\xi}}; \SP {v},{\xi}):= \Bigl(G_{,A}({\what{\xi}}) - G_{,A}({\xi})\Bigr)\bigl( {\what{v}}_i - {v}_i\bigr)
{\Phi}_{,i\alpha}({\what{F}}).
\end{equation}

\begin{remark}\label{NONSYM}
Note that \eqref{RENTDEFINTRO}, \eqref{RFLUXINTRO} are not symmetrical. Usually, in this type of calculations,
$({\what{v}},{\what{\xi}})$ denotes a non-smooth solution which is compared to the smooth $({v},{\xi})$. The definition
\eqref{RENTDEFINTRO} ensures that one evaluates gradient $\D\eta$ at the smooth solution to avoid computing the time derivative
at the shock (since $\D \eta$ appears in the identity for $\del_t \eta^r$). Also note that the definition of the relative flux
$q^r_{\alpha}$ is usually not known in advance and is simply a consequence of computations.
\end{remark}

\par\smallskip

\begin{lemma}\label{RENTIDSMOOTHLMM} Let $({\what{v}},{\what{\xi}})$ and $({v},{\xi})$ be smooth solutions of \eqref{EXTSYS}.
Then
\begin{equation}\label{RENTIDSMOOTH}
    \del_t \SP \eta^r + \del_{\alpha} \SP q^r_{\alpha} = Q
\end{equation}
where the term $Q$ is "quadratic" of the form
\begin{equation}\label{QDEFSMOOTH}
\begin{aligned}
Q({\what{v}},{\what{\xi}}; \SP {v},{\xi} ) &:= \partial_{\alpha}{v}_i \Bigl(G_{,A}({\what{\xi}})-G_{,A}({\xi})\Bigr)\Bigl({\Phi}_{,i\alpha}^A({\what{F}})-{\Phi}_{,i\alpha}^A({F})\Bigr)\\
&\quad  + \partial_{\alpha}(G_{,A}({\xi})) \Bigl({\Phi}_{,i\alpha}^A({\what{F}})-{\Phi}_{,i\alpha}^A({F})\Bigr) \Bigl({\what{v}}_i-{v}_i\Bigr)\\
&\quad  + \partial_{\alpha}{v}_i \Bigl(G_{,A}({\what{\xi}})-G_{,A}({\xi})-G_{,AB}({\xi})({\what{\xi}}-{\xi})_B
\Bigr){\Phi}_{,i\alpha}^A({F}).\\
\end{aligned}
\end{equation}
\end{lemma}

\begin{proof}
Since $({\what{v}},{\what{\xi}})$ is a smooth solution to \eqref{EXTSYS} we have
\begin{equation*}
\begin{aligned}
\del_t \SP \eta({\what{v}},{\what{\xi}}) &= \del_t \Bigl( \SP \tfrac{1}{2}|{\what{v}}|^2 + G({\what{\xi}}) \SP \Bigr)={\what{v}}_i \del_t {\what{v}}_i + G_{,A}({\what{\xi}}) \del_t
{\what{\xi}}_A\\
& ={\what{v}}_i \del_{\alpha} \Bigl( G_{,A}({\what{\xi}}) {\Phi}^A_{,i\alpha}({\what{F}}) \Bigr) + G_{,A}({\what{\xi}}) \del_{\alpha} \Bigl( {\Phi}_{,i\alpha}({\what{F}})
{\what{v}}_i\Bigr)\\
& =\del_{\alpha} \Bigl( {\what{v}}_i G_{,A}({\what{\xi}}) {\Phi}^A_{,i\alpha}({\what{F}}) \Bigr) + G_{,A}({\what{\xi}}) \Bigl(\del_{\alpha} \bigl(
{\Phi}_{,i\alpha}({\what{F}}) {\what{v}}_i\bigr) - {\Phi}_{,i\alpha}({\what{F}}) \SP \del_{\alpha} {\what{v}}_i\Bigr).
\end{aligned}
\end{equation*}
Since ${\what{F}}(\cdot,0)$ is a gradient, \eqref{EXTSYS}$_2$ ensures that it stays gradient. Hence, in view of the
null-Lagrangian property \eqref{NLPROPEXT}$_1$, we have
\begin{equation*}
\del_{\alpha} \bigl( {\Phi}_{,i\alpha}({\what{F}}) \SP {\what{v}}_i\bigr) = {\Phi}_{,i\alpha}({\what{F}}) \SP \del_{\alpha} {\what{v}}_i
\end{equation*}
and therefore
\begin{equation}\label{DETASMOOTH}
\begin{aligned}
\del_t \SP \eta({\what{v}},{\what{\xi}})  = \del_{\alpha} \Bigl( {\what{v}}_i G_{,A}({\what{\xi}}) {\Phi}^A_{,i\alpha}({\what{F}}) \Bigr).
\end{aligned}
\end{equation}

\par\smallskip

Next, using again \eqref{EXTSYS}, we get
\begin{equation*}
\begin{aligned}
& \del_t \Bigl( \eta({v},{\xi}) + \D\eta({v},{\xi})({\what{v}}-{v},{\what{\xi}}-{\xi}) \Bigr) \\[1pt]
& \quad =  \del_t \Bigl( \tfrac{1}{2}|{v}|^2+ G({\xi}) + {v}_i ({\what{v}}_i - {v}_i) +
G_{,A}({\xi})({\what{\xi}}-{\xi})_A
\Bigr)\\[2pt]
& \quad = \del_t {v}_i ({\what{v}}_i - {v}_i) + G_{,AB}({\xi}) \SP \del_t {\xi}_B \SP ({\what{\xi}}-{\xi})_A +
{v}_i
\del_t {\what{v}}_i + G_{,A}({\xi}) \SP \del_t {\what{\xi}}_A\\[2pt]
&\quad = \del_{\alpha} \Bigl( G_{,A}({\xi}) {\Phi}^A_{,i\alpha}({F}) \Bigr) ({\what{v}}_i - {v}_i) \\[2pt]
&\qquad + G_{,AB}({\xi}) \SP \del_{\alpha} \Bigl( {\Phi}^B_{,i\alpha}({F}){v}_i \Bigr)  ({\what{\xi}}-{\xi})_A \\[2pt]
&\qquad + {v}_i \SP \del_{\alpha} \Bigl( G_{,A}({\what{\xi}}) {\Phi}^A_{,i\alpha}({\what{F}}) \Bigr) \\[2pt]
&\qquad + G_{,A}({\xi}) \SP \del_{\alpha}\Bigl( {\Phi}^A_{,i\alpha}({\what{F}}) {\what{v}}_i \Bigr).\\
\end{aligned}
\end{equation*}
We now modify the above identity by transferring the spatial $\del_{\alpha}$ derivatives onto (what usually is a smooth) solution
$({v},{\xi})$. This approach and the null-Lagrangian property \eqref{NLPROPEXT}$_1$ lead us to
\begin{equation}\label{DRENTTMP1}
\begin{aligned}
& \del_t \Bigl( \eta({v},{\xi}) + \D\eta({v},{\xi})({\what{v}}-{v},{\what{\xi}}-{\xi}) \Bigr) \\[1pt]
&\quad = \del_{\alpha} \Bigl( G_{,A}({\xi}) \Bigr) {\Phi}^A_{,i\alpha}({F}) ({\what{v}}_i - {v}_i) \\[2pt]
&\qquad + G_{,AB}({\xi}) \SP {\Phi}^B_{,i\alpha}({F}) \SP \del_{\alpha} {v}_i \SP ({\what{\xi}}-{\xi})_A  \SP \\[2pt]
&\qquad + \del_{\alpha} \Bigl( {v}_i \SP G_{,A}({\what{\xi}}) {\Phi}^A_{,i\alpha}({\what{F}}) \Bigr) - \del_{\alpha}{v}_i \SP \Bigl( G_{,A}({\what{\xi}}) {\Phi}^A_{,i\alpha}({\what{F}}) \Bigr) \\[2pt]
&\qquad + \del_{\alpha}\Bigl( {\what{v}}_i G_{,A}({\xi}) {\Phi}^A_{,i\alpha}({\what{F}})  \Bigr) - \del_{\alpha} \Bigl( G_{,A}({\xi}) \Bigr) \SP {\Phi}^A_{,i\alpha}({\what{F}}) \SP {\what{v}}_i .\\
\end{aligned}
\end{equation}
Since $G_{,AB}=G_{,BA}$ the second term on the right hand side of \eqref{DRENTTMP1} satisfies
\begin{equation*}
\begin{aligned}
G_{,AB}({\xi}) \SP {\Phi}^B_{,i\alpha}({F}) \del_{\alpha} \SP {v}_i \SP ({\what{\xi}}-{\xi})_A  = \del_{\alpha} \SP
{v}_i \Bigl( G_{,AB}({\xi}) ({\what{\xi}}-{\xi})_B \SP \Bigr) {\Phi}^A_{,i\alpha}({F})
\end{aligned}
\end{equation*}
and thus rearranging the terms of \eqref{DRENTTMP1} we obtain
\begin{equation}\label{DRENTTMP2}
\begin{aligned}
& \del_t \Bigl( \eta({v},{\xi}) + \D\eta({v},{\xi})({\what{v}}-{v},{\what{\xi}}-{\xi}) \Bigr) \\[1pt]
&\quad = \del_{\alpha} \Bigl( G_{,A}({\xi}) \Bigr) \Bigl( {\Phi}^A_{,i\alpha}({F}) - {\Phi}^A_{,i\alpha}({\what{F}})  \Bigr) ({\what{v}}_i-{v}_i)\\[2pt]
 &\qquad  - \del_{\alpha}  {v}_i \Bigl( G_{,A}({\what{\xi}})  - G_{,A}({\xi}) -G_{,AB}({\xi}) ({\what{\xi}}-{\xi})_B\SP \Bigr) {\Phi}^A_{,i\alpha}({F})\\[2pt]
&\qquad  + \del_{\alpha}  {v}_i \Bigl( G_{,A}({\what{\xi}}) - G_{,A}({\xi}) \Bigr) \Bigl({\Phi}^A_{,i\alpha}({F})-{\Phi}^A_{,i\alpha}({\what{F}})\Bigr)\\[2pt]
 &\qquad   - \del_{\alpha} \Bigl( G_{,A}({\xi}) \Bigr) {\Phi}^A_{,i\alpha}({\what{F}}) {v}_i  - G_{,A}({\xi}) {\Phi}^A_{,i\alpha}({\what{F}}) \SP \del_{\alpha} {v}_i
 \\[2pt]
&\qquad  + \del_{\alpha} \Bigl( {v}_i \SP G_{,A}({\what{\xi}}) {\Phi}^A_{,i\alpha}({\what{F}})\Bigr)  + \del_{\alpha} \Bigl({\what{v}}_i G_{,A}({\xi}) {\Phi}^A_{,i\alpha}({\what{F}}) \Bigr).\\
\end{aligned}
\end{equation}
Recalling the definition of the term $Q$ we see that \eqref{DRENTTMP2} may be written as
\begin{equation}\label{DRENTTMP3}
\begin{aligned}
\del_t \Bigl( & \eta({v},{\xi})  - \D\eta({v},{\xi})({\what{v}}-{v},{\what{\xi}}-{\xi}) \Bigr) =\\
&\qquad -Q - \del_{\alpha} \Bigl( {v}_i  G_{,A}({\xi}) {\Phi}^A_{,i\alpha}({\what{F}}) \Bigr) \\
&\qquad  +  \del_{\alpha} \Bigl( {v}_i \SP G_{,A}({\what{\xi}}) {\Phi}^A_{,i\alpha}({\what{F}})\Bigr) \\
&\qquad  +  \del_{\alpha} \Bigr( {\what{v}}_i G_{,A}({\xi}) {\Phi}^A_{,i\alpha}({\what{F}})\Bigr) .\\
\end{aligned}
\end{equation}

\par\smallskip

Now, we combine \eqref{DETASMOOTH} with \eqref{DRENTTMP3} to get
\begin{equation}\label{RENTIDSMOOTHTMP}
\begin{aligned}
\del_t & \Bigl( \eta({\what{v}},{\what{\xi}}) - \eta({v},{\xi}) - \D\eta({v},{\xi})({\what{v}}-{v},{\what{\xi}}-{\xi}) \Bigr) \\
&-\del_{\alpha} \Bigl( {\what{v}}_i G_{,A}({\what{\xi}}) {\Phi}^A_{,i\alpha}({\what{F}}) \Bigr)- \del_{\alpha} \Bigl( {v}_i  G_{,A}({\xi}) {\Phi}^A_{,i\alpha}({\what{F}}) \Bigr) \\
&+  \del_{\alpha} \Bigl( {v}_i \SP G_{,A}({\what{\xi}}) {\Phi}^A_{,i\alpha}({\what{F}})\Bigr)+  \del_{\alpha} \Bigr( {\what{v}}_i G_{,A}({\xi}) {\Phi}^A_{,i\alpha}({\what{F}})\Bigr)=Q.\\
\end{aligned}
\end{equation}
Recalling \eqref{RFLUXINTRO} we obtain
\begin{equation}\label{RFLUXSMOOTHTMP1}
\begin{aligned}
&-\Bigl( {\what{v}}_i G_{,A}({\what{\xi}}) {\Phi}^A_{,i\alpha}({\what{F}}) \Bigr)- \del_{\alpha} \Bigl( {v}_i  G_{,A}({\xi}) {\Phi}^A_{,i\alpha}({\what{F}}) \Bigr) \\
&+\Bigl( {v}_i \SP G_{,A}({\what{\xi}}) {\Phi}^A_{,i\alpha}({\what{F}})\Bigr)+  \del_{\alpha} \Bigr( {\what{v}}_i G_{,A}({\xi})
{\Phi}^A_{,i\alpha}({\what{F}})\Bigr)=\\
&\qquad \Bigl(G({\what{\xi}})-G({\xi})\Bigr)({v}_i - {\what{v}}_i){\Phi}^A_{,i\alpha}({\what{F}})=-q^r_{\alpha}.   \\
\end{aligned}
\end{equation}
Then, \eqref{RENTIDSMOOTHTMP} and \eqref{RFLUXSMOOTHTMP1} imply the desired identity \eqref{RENTIDSMOOTH}.
\end{proof}

\par\smallskip

 The identity \eqref{RENTDEFINTRO} can be used to estimate the evolution of the
difference between the two solutions. In particular, one can show that the solution $({\what{v}},{\what{\xi}})$ stays close to
$({v},{\xi})$ as long as the initial data do. For this to be realized one would need the "quadratic" term $Q$ to have the following
property:
\begin{itemize}
\item[(GC)] If $M>0$ is the constant such that
\begin{equation}\label{MBOUND}
\sup_{(x,t)\in \torus^3 \times[0,T]} \biggl(\sum_{\alpha,i } |\del_{\alpha} {v}_i| +
\sum_{\alpha,A}\bigl|\del_{\alpha} (G_{,A}({\xi}))\bigr| \biggr) \leqslant M,
\end{equation}
then there holds
\begin{equation}\label{CONDNESGR}
\bigl|Q\bigl({\what{v}},{\what{\xi}}; \SP {v},{\xi} \bigr)\bigr| \leqslant C \SP \eta^r ({\what{v}},{\what{\xi}}; \SP {v}, {\xi})
\end{equation}
for some constant $C=C(M)>0$ independent of $({\what{v}},{\what{\xi}})$.
\end{itemize}
Indeed, if (GC) is satisfied then one can conclude via the Gronwall lemma that for each smooth solution
$({\what{v}},{\what{\xi}})$ to \eqref{EXTSYS} and fixed smooth solution $({v},{\xi})$ satisfying \eqref{MBOUND} there holds
\begin{equation}\label{ERRORESTGR}
\INTT \,   \Bigl[\eta^r \bigl({\what{v}},{\what{\xi}};\SP {v}, {\xi}\bigr)\Bigr](x,t)  \, dx  \SP \leqslant  e^{C(M)t}\SP \INTT
\, \Bigl[\eta^r \bigl({\what{v}},{\what{\xi}};\SP {v}, {\xi}\bigr) \Bigr](x,0) \, dx
\end{equation}
which yields the desired estimate (and guarantees uniqueness of the solution).

\par

Observe that, the inequality \eqref{CONDNESGR} does not hold in general and also $Q$ is not necessarily quadratic. One must
impose certain requirements on the stored energy ${W}=G\circ {\Phi}$ or more precisely on the function $G({\what{\xi}})$ to
satisfy \eqref{CONDNESGR}. On the first glance it seems that it is sufficient to require hypotheses (H2)-(H4) which handle various
growth condition and integrability on $\torus^3$. However, splitting the term $Q$ into two parts as
\begin{equation}
\begin{aligned}
&\!\!Q({\what{v}},{\what{\xi}}; \SP {v},{\xi} ) =\biggl[\partial_{\alpha}{v}_i \Bigl(G_{,A}({\what{\xi}})-G_{,A}({\xi})\Bigr)\Bigl({\Phi}_{,i\alpha}^A({\what{F}})-{\Phi}_{,i\alpha}^A({F})\Bigr)\biggr]\\
& \!\! + \biggl[ \partial_{\alpha}\Bigl(G_{,A}({\xi})\Bigr) \Bigl({\Phi}_{,i\alpha}^A({\what{F}})-{\Phi}_{,i\alpha}^A({F})\Bigr) \Bigl({\what{v}}_i-{v}_i\Bigr)\\
&\!\!\quad  + \partial_{\alpha}{v}_i \Bigl(G_{,A}({\what{\xi}})-G_{,A}({\xi})-G_{,AB}({\xi})({\what{\xi}}-{\xi})_B
\Bigr){\Phi}_{,i\alpha}^A({F}) \biggr] =: Q_1 + Q_2,
\end{aligned}
\end{equation}
we find that $Q_1$ (the last $10$ terms in its sum) fails to comply with (GC) regardless of (H2)-(H4).

\par\smallskip

To satisfy $(GC)$ there are two options to consider: \vspace{-3pt}
\begin{itemize}
\item[(O1)]  One can assume (H1) which implies (GC). The hypothesis (H1) is used in \cite{MT12} to handle the convergence in the
    semi-discrete case. The advantage of (H1) is that it allows to work with a very concrete class of functions. The disadvantage is
    that it restricts the class of stored energies even though allows for $L^p$ growth in ${\what{F}}$ component. The reviewer of
    \cite{MT12} noticed to me that perhaps it is best to work with more general class of stored energies.

\item[(O2)] One can impose the following requirement: if $({v},{\xi})$ satisfies \eqref{MBOUND} then
\begin{equation*}
Q_1({\what{v}},{\what{\xi}}; \SP {v}, {\xi}) \leqslant C_1 \SP \eta^r({\what{v}},{\what{\xi}},{v},{\xi})
\end{equation*}
for some $C_1=C_1(M)$ independent of $({\what{v}},{\what{\xi}})$. This requirement together with (H2)-(H5) implies $(GC)$
and allows for more general class of stored energies. The disadvantage, however, is that this requirement hides information about
the true nature of the stored energy.
\end{itemize}

\subsection{Fully-discrete scheme and approximates}

To mimic the approach what has been used in the case of semi-discrete settings \cite{MT12} we rewrite the scheme
\eqref{FDISCEXTSYSNONDIV_SHORT} componentwise using index notations: Given the $(n-1)$-st time iterate
\begin{equation*}
(v^{n-1}_h,\xi^{n-1}_h) = (v^{n-1}_h,
F^{n-1}_h,Z^{n-1}_h,w^{n-1}_h) \in U_h \times H^{F}_h \times
H^{Z}_h \times H^{w}_h
\end{equation*}
find the $n$-th iterate
\begin{equation*}
(v^{n}_h,\xi^{n}_h) = (v^{n}_h, F^{n}_h,Z^{n}_h,w^{n}_h) \in U_h
\times H^{F}_h \times H^{Z}_h \times H^{w}_h
\end{equation*}
by solving for $i=1,\dots,3$, $A=1,\dots,19$
\begin{equation}\label{FDSCHEMECOMP1}
\begin{aligned}
\INTT \SP \biggl(\frac{v^{n}_{i,\SP h} - v^{n-1}_{i,\SP
h}}{\tau}\biggr) \varphi^i_{h} \SP dx & = - \INTT \Bigl(
G_{,A}(\xi^n_h) \SP {\Phi}^A_{,i \alpha}(F^{n-1}_h) \SP \Bigl)
\del_{\alpha} \varphi^i_{h} \, dx
\\[3pt]
\INTT \biggl( \frac{\xi^n_{A,h} - \xi^{n-1}_{A,h}}{\tau}\biggr)
\psi_h \SP dx &= \INTT \SP  \Bigl( {\Phi}^A_{,i \alpha}(F^{n-1}_h)
\SP \del_{\alpha} v_{i,h}^n \Bigr) \psi^A_h \, dx
\end{aligned}
\end{equation}
for all $\varphi_h=(\varphi^i_h)_{i=1}^3 \in U_h$, $\psi_h=(\psi^A_{h})_{A=1}^{19} \in H_h := H^F_h \times H^Z_h \times
H^w_h$.

\par\smallskip

The choice of finite element spaces has a great impact on the fully discrete scheme. Namely, the space $H_h$ of test functions
turned out to be so rich ({\it c.f.} \eqref{CONSTRPROP3}) that the last equation in \eqref{FDSCHEMECOMP1} holds exactly, that is,
\eqref{FDSCHEMECOMP1} is equivalent to
\begin{equation}\label{FDSCHEMECOMP2}
\begin{aligned}
\INTT \SP \frac{v^{n}_{i,\SP h} - v^{n-1}_{i,\SP h}}{\tau} \SP
\varphi^i_{h} \SP dx & = - \INTT \Bigl(G_{,A}(\xi^n_h) \SP
{\Phi}^A_{,i \alpha}(F^{n-1}) \Bigl) \del_{\alpha} \varphi^i_{h} \SP
dx
\\[3pt]
\frac{\xi^n_{A,h} - \xi^{n-1}_{A,h}}{\tau}   \SP &= \SP {\Phi}^A_{,i
\alpha}(F^{n-1}_h) \SP \del_{\alpha} v_{i,h}^n
\end{aligned}
\end{equation}
for all $\varphi_h=(\varphi^i_h)_{i=1}^3 \in U_h$.

\par\smallskip

\begin{remark}
The scheme in the form \eqref{FDSCHEMECOMP2} provides a great opportunity to us since we are able to exploit the
null-Lagrangian properties. Namely, \eqref{FDSCHEMECOMP2}$_2$ guarantees that if $F^0_{h}$ is a gradient then $F^n_{h}$,
$n\geqslant 1$ are all gradients as well and hence the null-Lagrangian properties could be exploited regardless of $F^n$ being
discontinuous; see Lemma \ref{DIVNLLMM}. For example, when  ${y}(x,t)$ is a smooth map that induces initial data
\begin{equation*}
{v}_0(x)=\del_t \SP {y}(x,0), \quad {F}_0=\nabla
{y} (x,0).
\end{equation*}
we set
\begin{equation}\label{IDATAPROJ}
F^0_h:=\nabla (\PP^U {y}(x,0)) \in H^F
\end{equation}
where $\PP^U$ is a standard $L^2$-projector on $U_h$ defined in \eqref{ORTHPROJ}$_1$.
\end{remark}

\par\smallskip

\noindent{\bf Approximates.} Given the sequence of spatial iterates $(v^n_h,\xi^n_h)$, $n\geqslant 1$ we define  (following
\cite{MT12}) the time-continuous, piecewise linear interpolates
\begin{equation}
\what{\Theta}^{\SP(\tau,h)}:=\bigl(\what{v}^{\SP (\tau, h)},
\what{\xi}^{\SP(\tau, h)}\bigr)
\end{equation}
 with
\begin{equation}\label{CONTINTP}
\begin{aligned}
\what{v}^{\SP (\tau, h)}(t)   &= \sum^{\infty}_{n=1}\Chi^n(t) \Bigl(v^{n-1}_h+\frac{t-\tau(n-1)}{\tau}(v^n_h-v^{n-1}_h)\Bigr)\\
\what{\xi}^{\SP (\tau, h)}(t) &= \Bigl(\what{F}^{\SP(\tau,h)},\what{Z}^{\SP(\tau,h)},\what{w}^{\SP(\tau,h)}\Bigr)(t)\\
&=\sum^{\infty}_{n=1}\Chi^n(t)\Bigl( \xi^{n-1}_h+\frac{t-\tau(n-1)}{\tau}(\xi^n_h - \xi^{n-1}_h)\Bigr),\\
\end{aligned}
\end{equation}
and the piecewise constant interpolates
\[
\bar{\Theta}^{(h)}:=(\bar{v}^{\SP(\tau,h)}\,, \bar{\xi}^{\SP(\tau, h)}) \quad \text{and} \quad {\bar{F}_*}^{\SP (\tau,h)}
\]
 by
\begin{equation}\label{CONSTINTP}
\begin{aligned}
\bar{v}^{\SP(\tau,h)}(t)&=\sum^{\infty}_{n=1}\Chi^n(t)v^n_h\\
\bar{\xi}^{\SP(\tau,h)}(t)&=(\bar{F}^{\SP (\tau,h)},\bar{Z}^{(\SP \tau, h)},\bar{w}^{\SP(\tau, h)})(t)=\sum^{\infty}_{n=1}\Chi^n(t)\xi^n_h\\
{\bar{F}_*}^{\SP(\tau,h)}(t)&=\sum^{\infty}_{n=1}\Chi^n(t)F^{n-1}_h,
\end{aligned}
\end{equation}
where $\Chi^n(t)$ is the characteristic function of the interval $I_n:=[(n-1)\tau,n\tau)$. Notice that ${\bar{F}_*}^{\SP(\tau,h)}$ is
the time-shifted version of $F^{\SP(\tau,h)}$  and used later in various calculations.

\par\medskip

\noindent{\bf The scheme via approximates.} Clearly the linear approximates \eqref{CONTINTP} are absolutely continuous in time.
This motivates to rewrite the discrete system \eqref{FDSCHEMECOMP2} in terms of the approximates \eqref{CONTINTP},
\eqref{CONSTINTP}. Then the scheme \eqref{FDSCHEMECOMP2} transforms into:
\begin{equation}\label{FDSCHEMECOMPAPP}
\begin{aligned}
\INTT \SP \Bigl(\del_t \what{v}^{\SP(\tau,h)}_i\Bigr) \SP \varphi^i_{h} \SP dx &
= - \INTT \Bigl(G_{,A}\bigl(\bar{\xi}^{\SP(\tau,h)}\bigr) \SP {\Phi}^A_{,i \alpha}\bigl({\bar{F}_*}^{\SP(\tau,h)}\bigr)\Bigr) \SP \del_{\alpha} \varphi^i_{h} \, dx \\[3pt]
\del_t \what{\xi}^{\SP(\tau,h)} &= \SP {\Phi}^A_{,i \alpha}\bigl({\bar{F}_*}^{\SP(\tau,h)}\bigr) \SP \del_{\alpha}
\bar{v}^{\SP(\tau,h)}_i
\end{aligned}
\end{equation}
for a.e. $t>0$ and $\forall\varphi_h=(\varphi^i_h)_{i=1}^3 \in U_h$.

\subsection{Stability of the fully-discrete variational scheme}

For the rest of the sequel, we {\it suppress the dependence on $\tau,h$} to simplify notations and assume that:
\begin{itemize}
\item[{{(A1)}}] $\what{\Theta}=({\what{v}},{\what{\xi}})$ are the time-continuous approximates; see \eqref{CONTINTP}.

\item[{{(A2)}}] $\bar{\Theta}=(\bar{v},\bar{\xi})$ and ${\bar{F}_*}$ are the constant approximates; see \eqref{CONSTINTP}.

\item[{(A3)}] ${\Theta}=({v},{\xi})=({v},{F},Z,w)$ is a smooth solution to \eqref{EXTSYS} on $\mathbb{T}^3 \times [0,T]$.

\item[{(A4)}]  $F^0$, ${F}(\cdot,0)$ are gradients and initial iterate $(v^0,\xi^0)\in U_h \times H_h$.

\end{itemize}

\par\smallskip

The goal of this section is to derive an identity for a relative energy among the two solutions. To this end, we define the relative
entropy
\begin{equation}\label{RENTDEF}
\eta^r(\what{\Theta},{\Theta}):=\eta(\what{\Theta})-\eta({\Theta})-\D\eta({\Theta})(\what{\Theta}-{\Theta})
\end{equation}
where $\eta$ is the convex entropy of the extended elasticity system \eqref{EXTSYS} defined by
\begin{equation}\label{ENTDEF2}
   \eta(\Theta)=\frac{1}{2} |v|^2 + G(\xi),\quad \Theta=(v,\xi).
\end{equation}

\par\smallskip

To deriving the relative entropy identity, we will employ the lemma \ref{DIVNLLMM} that extends the null-Lagrangian properties to
non-smooth gradients. These properties will be used extensively in our computations throughout the paper.

\vspace{3pt}

\begin{lemma}[\textbf{relative entropy identity}]\label{RENTIDLMM} Let (A1)-(A4) hold. Then
\begin{equation}\label{RENTID}
\INTT \SP \partial_t \SP \eta^r(x,t) \, dx = \INTT \biggl( -\frac{1}{\tau} D + Q+E+\bar{E} \biggr) \SP dx, \quad
\mbox{a.e.} \;\; t \in[0,T]
\end{equation}
where
\begin{equation}\label{QTERM}
\begin{aligned}
Q &:=     \partial_{\alpha}(G_{,A}({\xi})) \bigl({\Phi}_{,i\alpha}^A({\what{F}})-{\Phi}_{,i\alpha}^A({F})\bigr)
\bigl({\what{v}}_i-{v}_i\bigr)\\
&\qquad  + \partial_{\alpha}{v}_i
\bigl(G_{,A}({\what{\xi}})-G_{,A}({\xi})\bigr)\bigl({\Phi}_{,i\alpha}^A({\what{F}})-{\Phi}_{,i\alpha}^A({F})\bigr)\\
&\qquad  + \partial_{\alpha}{v}_i \Bigl(G_{,A}({\what{\xi}})-G_{,A}({\xi})-G_{,AB}({\xi})({\what{\xi}}-{\xi})_B
\Bigr){\Phi}_{,i\alpha}^A({F})\\
\end{aligned}
\end{equation}
estimates the difference between the two solutions,
\begin{equation}\label{DTERM}
D:=\sum_{n=1}^{\infty}\chi^n(t)D^n \quad \mbox{with} \quad
D^n:=\bigl(\nabla{\eta}({\bar{\Theta}})-\nabla{\eta}(\what{\Theta})\bigr) \delta\Theta^n,
\end{equation}
where
\begin{equation}\label{DITER}
\begin{aligned}
\delta \Theta^n &=(\delta v^n, \delta F^n, \delta Z^n, \delta w^n):=\Theta^n - \Theta^{n-1}\\
 &\quad =(v^n-v^{n-1}, F^n-F^{n-1}, Z^n-Z^{n-1}, w^n-w^{n-1}),\\
\end{aligned}
\end{equation}
is the dissipative term,
\begin{equation}\label{ETERM}
\begin{aligned}
E &:= \partial_{\alpha}(G_{,A}({\xi}))\Bigl[ \,{\Phi}_{,i\alpha}^A({F})               \bigl({\bar{v}}_i-{\what{v}}_i\bigr)\\
&\!\quad\qquad\qquad\qquad+  \bigl({\Phi}_{,i\alpha}^A({\what{F}})-{\Phi}_{,i\alpha}^A({F})\bigr)    \bigl({\bar{v}}_i-{\what{v}}_i\bigr)\\
&\!\quad\qquad\qquad\qquad+ \bigl({\Phi}_{,i\alpha}^A({\bar{F}_*})-{\Phi}_{,i\alpha}^A({\what{F}})\bigr)   \bigl({\bar{v}}_i-{\what{v}}_i\bigr)\\
&\!\quad\qquad\qquad\qquad+ \bigl({\Phi}_{,i\alpha}^A({\bar{F}_*})-{\Phi}_{,i\alpha}^A({\what{F}})\bigr)   \bigl({\what{v}}_i-{v}_i\bigr) \Bigr]\\
&\qquad+\partial_{\alpha}{v}_i\Bigl[ \bigl(G_{,A}({\bar{\xi}})-G_{,A}({\what{\xi}}) \bigr){\Phi}_{,i\alpha}^A({F})\\
&\!\!\!\qquad\qquad\qquad + \bigl(G_{,A}({\bar{\xi}})-G_{,A}({\what{\xi}})\bigr)        \bigl({\Phi}_{,i\alpha}^A({\bar{F}_*})-{\Phi}_{,i\alpha}^A({\what{F}})\bigr)\\
&\!\!\!\qquad\qquad\qquad + \bigl(G_{,A}({\bar{\xi}})-G_{,A}({\what{\xi}})\bigr)        \bigl({\Phi}_{,i\alpha}^A({\what{F}})-{\Phi}_{,i\alpha}^A({F})\bigr)  \\
&\!\!\!\qquad\qquad\qquad + \bigl(G_{,A}({\what{\xi}})-G_{,A}({\xi})\bigr)
\bigl({\Phi}_{,i\alpha}^A({\bar{F}_*})-{\Phi}_{,i\alpha}^A({\what{F}})\bigr)\Bigr]
\end{aligned}
\end{equation}
is the error term that appears due to the discretization in time, and
\begin{equation}\label{EBARTERM}
\bar{E}: \SP = \SP G_{,A}({\bar{\xi}}) \SP {\Phi}^A_{i\alpha}({\bar{F}_*}) \SP \del_{\alpha} \bigl((\PP^U {v})_i-{v}_i
\bigr)
\end{equation}
is the error term that appears due to spatial discretization.
\end{lemma}

\begin{proof}

By (A1) we have that $F^0_h$, ${F}(\cdot,0)$ are gradients. Hence by \eqref{FDSCHEMECOMP2}, \eqref{CONTINTP},
\eqref{CONSTINTP}, and the property (E1) we conclude
\begin{equation}\label{FGRADPROP}
\mbox{${\what{F}}$, ${\bar{F}}$, ${\bar{F}_*}$ and ${F}$ are gradients
$\forall t \in [0,T]$}.
\end{equation}

\par\smallskip

Next, recalling \eqref{ENTDEF2} we compute 
\begin{equation}\label{DETAAP}
\begin{aligned}
\!\partial_t \bigl(\eta(\what{\Theta})\bigr)
&      =   {\what{v}}_i \hspace{1pt} \partial_t{{\what{v}}_i}+G_{,A}({\what{\xi}}) \hspace{1pt} \partial_t{\what{\xi}}_A\\
&= {\bar{v}}_i \SP \del_t {\what{v}}_i + G_{,A}({\bar{\xi}}) \SP \del_t {\what{\xi}} + \frac{1}{\tau}\sum_{n=1}^{\infty} \Chi^n(t) \bigl(\nabla{\eta(\what{\Theta})}-\nabla{\eta({\bar{\Theta}})}\bigr)\hspace{1pt} \delta\Theta^n. \\
\end{aligned}
\end{equation}
By construction ${\bar{v}}(\cdot,t) \in U_h$, $\forall t\in[0,t]$. Thus, setting $\varphi = {\bar{v}}(\cdot,t)$ in the weak formulation
\eqref{FDSCHEMECOMPAPP}$_1$ and using \eqref{FDSCHEMECOMPAPP}$_2$ we obtain
\begin{equation*}
\INTT \SP {\bar{v}}_i \SP \del_t {\what{v}}_i \, dx = - \INTT  G_{,A}({\bar{\xi}})
\SP {\Phi}^A_{,i \alpha}({\bar{F}_*}) \SP \del_{\alpha} {\bar{v}}_i \, dx = -
\INTT  G_{,A}({\bar{\xi}}) \SP \del_t {\what{\xi}} \, dx.
\end{equation*}
Then, integrating expression \eqref{DETAAP} we obtain
\begin{equation}\label{INTDETAAP}
\begin{aligned}
\INTT \SP \partial_t \bigl(\eta(\what{\Theta})\bigr) \,  dx &=  \INTT \SP \biggl(-\frac{1}{\tau}\sum_{n=1}^{\infty}
\Chi^n(t) \bigl(\nabla{\eta({\bar{\Theta}})}-\nabla{\eta(\what{\Theta})}\bigr) \SP \delta\Theta^n \biggr) dx.
\end{aligned}
\end{equation}

\par\smallskip

Next, we compute
\begin{equation}\label{TMPRENT1}
\begin{aligned}
&\partial_t \Bigl( \tfrac{1}{2}{v}^2 + G({\xi})
+{v}_i({\what{v}}_i-{v}_i) + G_{,A}({\xi})({\what{\xi}}-{\xi})_A\Bigr)  =\\
&\qquad \partial_t{v}_i \SP ({\what{v}}_i-{v}_i)
+  \partial_t \bigl(G_{,A}({\xi})\bigr)({\what{\xi}}-{\xi})_A
+ {v}_i \SP \partial_t {\what{v}}_i + G_{,A}({\xi}) \SP \partial_t {\what{\xi}}_A.\\
\end{aligned}
\end{equation}
Since $({v},{\xi})$ is a smooth solution to \eqref{EXTSYS} we have
\begin{equation*}\label{TMPRENT2}
\begin{aligned}
\del_t {v}_i \SP &= \SP \del_{\alpha} \bigl( G_{,A}({\xi}) \SP {\Phi}^A_{,i\alpha}({F}) \bigr)
                 \SP = \SP \del_{\alpha} \bigl( G_{,A}({\xi}) \bigr) \SP {\Phi}^A_{,i\alpha}({F})\\[2pt]
\del_t {\xi}_A \SP &= \SP  \del_{\alpha} \bigl( {\Phi}^A_{,i\alpha}({F})\SP{v}_i \bigr) \SP = \SP
{\Phi}^A_{,i\alpha}({F}) \SP \del_{\alpha} {v}_i
\end{aligned}
\end{equation*}
where we used \eqref{NLPROPEXT}$_1$ and the fact that ${F}$ is a gradient. Also, from \eqref{FDSCHEMECOMP2} and
\eqref{FDSCHEMECOMPAPP} it follows that $\del_t {\what{v}} \in U^h$. Hence, since $\PP^U$ is a standard orthogonal projector,
by \eqref{FDSCHEMECOMPAPP}$_{1}$ we have
\begin{equation*}\label{TMPRENT3}
    \INTT {v}_i \SP \del_t {\what{v}}_i \SP dx =     \INTT \bigl(\PP^U{v}\bigl)_i \DSP \del_t {\what{v}}_i \, dx
    =- \INTT \SP  G_{,A}({\bar{\xi}}) {\Phi}^A_{i\alpha}({\bar{F}_*}) \SP \del_{\alpha} (\PP^U {v})_i \, dx.
\end{equation*}
Finally, by \eqref{FDSCHEMECOMPAPP}$_2$
\begin{equation*}\label{TMPRENT4}
   \INTT \SP G_{,A}({\xi}) \SP \del_t {\what{\xi}}_A \, dx = \INTT
   G_{,A}({\xi})\SP
    {\Phi}_{,i\alpha}({\bar{F}_*}) \SP \del_{\alpha} {\bar{v}}_i \, dx.
\end{equation*}
Now, integrating \eqref{TMPRENT1} and using the last four identities we get
\begin{equation}\label{DETABARPLUSGRAD}
\begin{aligned}
& \INTT \SP  \del_t \bigl( \eta({\Theta}) + \nabla\eta({\Theta}) ( \what{\Theta}-{\Theta}) \bigr) \SP dx
 = \\
 & \qquad  \INTT \SP \SP \del_{\alpha} \bigl( G_{,A}({\xi}) \bigr) \SP {\Phi}^A_{,i\alpha}({F})\SP
 ({\what{v}}_i-{v}_i) \, dx  \\
 &\quad + \INTT \SP G_{,AB}({\xi}) \SP {\Phi}^B_{,i\alpha}({F}) \SP \del_{\alpha} {v}_i ({\what{\xi}}-{\xi})_A \, dx\\
 &\quad + \INTT  \Bigl( G_{,A}({\xi})\SP {\Phi}^A_{,i\alpha}({\bar{F}_*}) \SP \del_{\alpha} {\bar{v}}_i
        - G_{,A}({\bar{\xi}}) \SP {\Phi}^A_{i\alpha}({\bar{F}_*}) \SP \del_{\alpha} {v}_i \Bigr) \, dx\\
 &\quad + \INTT \SP  G_{,A}({\bar{\xi}}) \SP {\Phi}^A_{i\alpha}({\bar{F}_*}) \DSP \del_{\alpha} \bigl({v}_i - (\PP^U {v})_i \bigr) \,
 dx.
\end{aligned}
\end{equation}

\par\smallskip

Subtracting \eqref{DETABARPLUSGRAD} from \eqref{DETAAP}, recalling \eqref{RENTDEF}, \eqref{DTERM} and \eqref{EBARTERM},
and then using the fact that $G_{,AB}=G_{BA}$ we conclude that
\begin{equation}\label{RELENTID1}
\begin{aligned}
& \INTT \SP  \eta^r(x,t) \, dx = \INTT \SP \biggl( -\frac{1}{\tau} \sum_{n=1}^{\infty} \Chi^n(t)D_n + \bar{E} +
J\biggr) dx
\end{aligned}
\end{equation}
with
\begin{equation}\label{JDEF}
\begin{aligned}
  J&:=      \del_{\alpha} {v}_i \Bigl( G_{,A}({\what{\xi}})- G_{,A}({\xi})- G_{,AB}({\xi}) \SP  ({\what{\xi}}-{\xi})_B \Bigr) {\Phi}^A_{,i\alpha}({F})\\[1pt]
 &\quad +  \del_{\alpha} {v}_i \Bigl[ \bigl( G_{,A}({\bar{\xi}}) - G_{,A}({\xi})\bigr) \SP {\Phi}^A_{i\alpha}({\bar{F}_*}) - \bigl(G_{,A}({\what{\xi}}) - G_{,A}({\xi}) \bigr) \SP {\Phi}^A_{,i\alpha}({F})\Bigr] \\[1pt]
 &\quad +  \del_{\alpha} ( G_{,A}({\xi}) ) \Bigl( {\Phi}^A_{,i\alpha}({\bar{F}_*})\SP ({\bar{v}}_i-{v}_i) -{\Phi}^A_{,i\alpha}({F})\SP ({\what{v}}_i-{v}_i) \Bigr)\\[1pt]
 &\quad +  G_{,A}({\xi})\SP {\Phi}^A_{,i\alpha}({\bar{F}_*}) \SP \del_{\alpha} ({v}_i-{\bar{v}}_i) +  \del_{\alpha} ( G_{,A}({\xi})) \SP {\Phi}^A_{,i\alpha}({\bar{F}_*})\SP ({v}_i-{\bar{v}}_i)
 \\[3pt]
 &\,\quad := J_1+J_2+J_3+J_4.\\
\end{aligned}
\end{equation}

\par\smallskip

Consider the terms on the right-hand side of \eqref{JDEF}. First, we  rearrange the term $J_2$ as follows
\begin{equation}\label{J2}
\begin{aligned}
J_2&= \partial_{\alpha}{v}_i
\Bigl[ \bigl(G_{,A}({\bar{\xi}})- G_{,A}({\xi})\bigr) \hspace{1pt} {\Phi}^A_{,i\alpha}({\bar{F}_*}) - \bigl(G_{,A}({\what{\xi}})- G_{,A}({\xi})\bigr) \hspace{1pt}  {\Phi}^A_{,i\alpha}({F})\Bigr]\\[0pt]
&= \partial_{\alpha}{v}_i \Bigl[ \bigl(G_{,A}({\bar{\xi}})- G_{,A}({\what{\xi}})\bigr) \bigl( {\Phi}^A_{,i\alpha}({\bar{F}_*}) -
{\Phi}^A_{,i\alpha}({\what{F}}) \bigr ) \\
&\!\qquad\qquad + \bigl(G_{,A}({\bar{\xi}})- G_{,A}({\what{\xi}})\bigr) \bigl(
{\Phi}^A_{,i\alpha}({\what{F}}) - {\Phi}^A_{,i\alpha}({F}) \bigr )\\
&\!\qquad\qquad + \bigl(G_{,A}({\bar{\xi}})- G_{,A}({\what{\xi}})\bigr) \hspace{1pt}
{\Phi}^A_{,i\alpha}({F})\\
&\!\qquad\qquad + \bigl(G_{,A}({\what{\xi}})- G_{,A}({\xi})\bigr) \bigl(
{\Phi}^A_{,i\alpha}({\bar{F}_*}) - {\Phi}^A_{,i\alpha}({\what{F}}) \bigr ) \\
&\!\qquad\qquad + \bigl(G_{,A}({\what{\xi}})- G_{,A}({\xi})\bigr) \bigl( {\Phi}^A_{,i\alpha}({\what{F}}) -
{\Phi}^A_{,i\alpha}({F}) \bigr ) \Bigr].
\end{aligned}
\end{equation}
Then, we modify the term $J_3$ writing it in the following way:
\begin{equation}\label{J3}
\begin{aligned}
J_3 & = \partial_{\alpha}(G_{,A}({\xi})) \Bigl[ {\Phi}_{,i\alpha}^A({\bar{F}_*}) ({\bar{v}}_i-{v}_i) -{\Phi}_{,i\alpha}^A({F})({\what{v}}_i-{v}_i) \Big]\\[0pt]
&     = \partial_{\alpha}(G_{,A}({\xi})) \Bigl[ \bigl({\Phi}_{,i\alpha}^A({\what{F}})-{\Phi}_{,i\alpha}^A({F}) \bigr) \bigl({\what{v}}_i-{v}_i\bigr) \\
&\! \quad \qquad\qquad\qquad  + \bigl({\Phi}_{,i\alpha}^A({\bar{F}_*})-{\Phi}_{,i\alpha}^A({\what{F}}) \bigr) \bigl({\what{v}}_i- {v}_i\bigr)\\
&\! \quad \qquad\qquad\qquad  + \bigl({\Phi}_{,i\alpha}^A({\bar{F}_*})-{\Phi}_{,i\alpha}^A({\what{F}}) \bigr) \bigl({\bar{v}}_i- {\what{v}}_i\bigr) \\
&\!\quad \qquad\qquad\qquad   + \bigl({\Phi}_{,i\alpha}^A({\what{F}})-{\Phi}_{,i\alpha}^A({F}) \bigr) \bigl({\bar{v}}_i- {\what{v}}_i\bigr)\\
&\! \quad \qquad\qquad\qquad  + {\Phi}_{,i\alpha}^A({F}) \bigl({\bar{v}}_i- {\what{v}}_i\bigr) \Bigr].
\end{aligned}
\end{equation}

\par\smallskip

Now, we consider the term $J_4$. We have to exploit the {\it null-Lagrangian} structure incorporated in the scheme to handle it
since $J_4$ is linear in ${v}-{\bar{v}}$. This, in general, would cause difficulties in making the error estimate via the relative
entropy method.

\par\smallskip

Observe that  $({v}-{\bar{v}})(\cdot,t)\in L^{\infty}(\torus^3 \SP; \RR^3)$ and ${\bar{F}_*}(\cdot,t) \in L^{\infty}(\torus^3 \SP;
\mdd{3})$ is a gradient  $\forall t\in[0,T]$. Then, by Lemma \ref{DIVNLLMM}, for each fixed $i=1,\dots,3$ and $A=1,\dots,19$, there
holds
\begin{equation}\label{J4T1}
\sum_{\alpha=1}^3 \partial_{\alpha}\biggl(\pd{{\Phi}^A}{{F}_{i\alpha}}({\bar{F}_*}) \SP ({v}_i-{\bar{v}}_i) \biggr) =
\sum_{\alpha=1}^3 \pd{{\Phi}^A}{{F}_{i\alpha}}({\bar{F}_*})\SP\partial_{\alpha}({v}_i-{\bar{v}}_i) \;\;\, \mbox{in} \;\;\, \cD'(\mathbb{T}^3)\\
\end{equation}
in the sense that
\begin{equation}\label{J4T2}
\begin{aligned}
&-\INTT \sum_{\alpha=1}^3\biggl(\pd{{\Phi}^A}{{F}_{i\alpha}}({\bar{F}_*}) ({v}_i-{\bar{v}}_i) \biggr) \partial_{\alpha}\phi \,
dx =\\
&\quad\qquad\qquad \INTT  \sum_{\alpha=1}^3 \biggl( \pd{{\Phi}^A}{{F}_{i\alpha}}({\bar{F}_*})\SP\partial_{\alpha}({v}_i-{\bar{v}}_i) \biggr) \phi \, dx, \quad \forall \phi \in C^{\infty}(\torus^3).\\
\end{aligned}
\end{equation}
By the standard density argument \eqref{J4T2} holds for all $\phi\in W^{1,1}(\torus^3)$. Since ${\xi}$ is smooth, we have
$\phi_A:=G_{,A}({\xi}(\cdot,t))\in C^1(\torus^3)$. Then, considering \eqref{J4T2} with $\phi_A$ for each $A=1,\dots,19$ and
summing over $i=1,\dots,3$ and $A=1,\dots,19$  (using the summation convention over repeated indices) we conclude
\begin{equation}\label{J4INT}
\begin{aligned}
&0=\INTT \SP {\Phi}^A_{,i\alpha}({\bar{F}_*}) ({v}_i-{\bar{v}}_i) \SP \partial_{\alpha} \bigl(G_{,A}({\xi})\bigr) \SP
dx  \\
&\qquad + \INTT  \SP {\Phi}^A_{,i\alpha}({\bar{F}_*})\SP\partial_{\alpha}({v}_i-{\bar{v}}_i) \SP G_{,A}({\xi}) \SP dx = \INTT \SP J_4 \SP dx.\\
\end{aligned}
\end{equation}

\par

Finally, combining \eqref{JDEF}-\eqref{J3} and \eqref{J4INT} we conclude
\begin{equation}\label{JINT}
\INTT J \, dx  = \INTT \bigl( J_1 + J_2 + J_3 +J_4 \bigr) \, dx = \INTT \bigl( Q+E \bigr) \, dx
\end{equation}
with terms $Q$, $E$ defined in \eqref{QTERM}, \eqref{ETERM} respectively. Then \eqref{RELENTID1}, \eqref{JINT} imply the
desired identity \eqref{RENTID}.
\end{proof}

\input{bibliography} 

\end{document}

%% file: my_macros.tex
\newcommand{\T}{\top}
\newcommand{\trp}{^{\text{T}}}
\newcommand{\diag}{{\text{diag}}}
\newcommand{\Def}{\mathrm{Def}}
\newcommand{\f}[2]{\frac{#1}{#2}}
\newcommand{\gr}[1]{\nabla {#1}}

\def\mc{\mathcal}
\def\com#1{\quad\text{#1}\quad}

\newcommand{\wint}{(w*)\dash \int}

\newcommand{\Id}{{\bf I}}

\def\Weak{\,\,\relbar\joinrel\rightharpoonup\,\,}

\newcommand{\mat}{\hbox{Mat}}
\def\del{\partial}
\def\torus{\mathbb{T}}
\def\Real{{\mathop{\hbox{\msym \char '122}}}}
\newcommand{\n}{\mbox{\boldmath $ \nu$}}
\newcommand{\DD}{\mathbb{D}}
\newcommand{\D}{\mbox{D}}
\newcommand{\cd}{{\cal D}}
\newcommand{\sma}{_{{ A}}}
\newcommand{\ce}{{\cal E}}
\newcommand{\dist}{{\mbox{dist}}}
\newcommand{\cs}{{\cal S}}
\newcommand{\ca}{{\cal A}}
\newcommand{\cc}{{\cal C}}
\newcommand{\cp}{{\cal P}}
\newcommand{\cb}{{\cal B}}
\newcommand{\cl}{{\cal L}}
\newcommand{\cg}{{\cal G}}
\newcommand{\hj}{^{J}}
\newcommand{\m}{\mbox{\boldmath $ \mu$}}
\newcommand{\nutx}{\nu_{t,x}}
\newcommand{\mutx}{\mu_{t,x}}
\newcommand{\hjm}{^{J-1}}
\newcommand{\cu}{{\cal U}}
\newcommand{\tn}{{\tilde\|}}
\newcommand{\rbar}{\overline{r}}
\newcommand{\dxdt}{\; dxdt}
\newcommand{\dx}{\; dx}
\newcommand{\oeps}{\overline{\varepsilon}}
\newcommand{\cgl}{\hbox{Lie\,}{\cal G}}
\newcommand{\ttd}{{\tt d}}
\newcommand{\ttdel}{{\tt \delta}}
\newcommand{\ns}{\nabla_*}
\newcommand{\csl}{{\cal SL}}
\newcommand{\spsi}{_{{{\Psi}}}}
\newcommand{\kr}{\hbox{Ker}}
\newcommand{\qinfo}{\stackrel{\circ}{Q}_{\infty}}
\newcommand{\eq}[1]{\begin{equation*}#1\end{equation*}}
\newcommand{\diagmtx}[3]{\begin{bmatrix} #1 \; \quad \; \quad\\\quad \; #2 \; \quad \\\quad \; \quad \; #3\end{bmatrix}}

\newcommand{\B}{\mathcal{B}}
\newcommand{\BB}{\mathfrak{B}}
\newcommand{\cof}{\mathrm{cof\, }}
\newcommand{\RR}{\mathbb{R}}
\newcommand{\eps}{\varepsilon}
\newcommand{\Chi}{\mathcal{X}}

\newcommand{\cD}{\mathcal{D}}
\newcommand{\pd}[2]{\frac{\partial #1}{\partial #2}}
\newcommand{\BBR}[1]{\left( #1 \right)}
\newcommand{\BBS}[1]{\left[ #1 \right]}
\newcommand{\BBF}[1]{\left\{ #1 \right\}}
\newcommand{\BBN}[1]{\left\| #1 \right\|}
\newcommand{\BBA}[1]{\left | #1 \right |}
\newcommand{\mdd}[1]{M^{#1\times #1}}
\newcommand{\mddplus}[1]{M^{#1\times #1}_+}
\newcommand{\mddplusb}[1]{{\bar{M}}^{#1\times #1}_+}
\newcommand{\PHD}{\frac{\partial{\varPhi^A}}{\partial{F_{i\alpha}}}}
\newcommand{\PSD}{\frac{\partial{\Phi^A}}{\partial{F_{i\alpha}}}}

\newcommand{\txq}[1]{\quad \mbox{#1} \quad}
\newcommand{\txc}[1]{\; \mbox{#1} \;}
\newcommand{\txs}[1]{\: \mbox{#1} \:}
\newcommand{\tx}{\mbox}
\newcommand{\sbt}{\subset}
\newcommand{\ssbt}{\subset\subset}

\newcommand{\tld}[1]{\widetilde{#1}}

\newcommand{\Var}{\mathbb{V}}

\def\Weak{\,\,\relbar\joinrel\rightharpoonup\,\,}
\def\Real{{\mathop{\hbox{\msym \char '122}}}}


\newcommand{\SP}{\hspace{1pt}}
\newcommand{\DSP}{\hspace{2pt}}


%
%
\newtheorem{theorem}{Theorem}[section]
\newtheorem*{proposition*}{Proposition}
\newtheorem{proposition}[theorem]{Proposition}
\newtheorem*{theorem*}{Theorem}
\newtheorem{corollary}[theorem]{Corollary}
\newtheorem{convention}[theorem]{Convention}
\newtheorem*{convention*}{Convention}
\newtheorem*{corollary*}{Corollary}
\newtheorem*{definition*}{Definition}
\newtheorem{definition}[theorem]{Definition}
\newtheorem{lemma}[theorem]{Lemma}
\newtheorem*{lemma*}{Lemma}
\newtheorem{remark}[theorem]{Remark}
\newtheorem*{remark*}{Remark}
\newtheorem*{mtheorem*}{Main Theorem}
\newtheorem{mtheorem}[theorem]{Main Theorem}

\newcommand{\vertiii}[1]{{\left\vert\kern-0.25ex\left\vert\kern-0.25ex\left\vert #1
\right\vert\kern-0.25ex\right\vert\kern-0.25ex\right\vert}}

\newcommand{\ws}{\text{\it w*}}

\newcommand{\dash}{\mbox{-}}

\newcommand{\<}{\big\langle}

\renewcommand{\>}{\big\rangle}


%% file: disc_var_scheme.bbl
\begin{thebibliography}{99}
\setlength{\parskip}{0.6em}

\footnotesize

\bibitem{Ball77} {\sc J.M. Ball}, Convexity conditions and existence theorems in nonlinear elasticity, {\em Arch.
Rational Mech. Anal.} {\bf 63} (1977), 337-403.

\bibitem{Ball82}
{\sc  J.M. Ball},{ Discontinuous Equilibrium Solutions and
Cavitation in Nonlinear Elasticity, {\itshape Phil. Trans. of the
Royal Society of London. Series A, Math. and Phys. Sci.,}
{\bfseries 306} (1982) 557-611. }

\bibitem{BCO}
{\sc J.M. Ball, J.C. Currie, P.J. Olver}, Null Lagrangians, weak continuity and variational problems of arbitrary
order, {\em J. Funct. Anal.} {\bf 41} (1981), 135-174.

\bibitem{Ball02}
{\sc J.M. Ball}, Some open problems in elasticity. In {\em Geometry, mechanics, and dynamics.} P. Newton, Ph. Holmes
and A. Weinstein, eds; Springer-Verlag, New York (2002), pp 3-59.


\bibitem{Ciarlet}
{\sc P.G. Ciarlet}, {Mathematical Elasticity}, Vol. {\bf 1}, North Holland, (1993).

\bibitem{Dafermos79}
{\sc C. Dafermos}, The second law of thermodynamics and stability, {\it Arch. Rational Mech. Anal.} {\bf 70} (1979),
167-179.

\bibitem{Dafermos86}
{\sc C. Dafermos}, Quasilinear hyperbolic systems with involutions, {\it Arch. Rational Mech. Anal.} {\bf 94} (1986),
373-389.

\bibitem{Dafermos10}
{\sc C. Dafermos},  {\it Hyperbolic conservation laws in continuum physics.} Third edition. Grundlehren der
Mathematischen Wissenschaften, 325. Springer-Verlag, Berlin, (2010).

\bibitem{DafermosHrusa85}
{\sc C. Dafermos and W. Hrusa}, Energy methods for quasilinear hyperbolic initial-boundary value problems. Applications
to elastodynamics. {\it Arch. Rational Mech. Anal.} {\bf 87} (1985), 267-292.

\bibitem{DST}
{ \sc S. Demoulini, D.M.A. Stuart, A.E.Tzavaras}, A variational approximation scheme for three dimensional
elastodynamics with polyconvex energy, {\itshape Arch. Rat. Mech. Anal.} {\bfseries 157} (2001),  325--344.

\bibitem{DST2}
{ \sc S. Demoulini, D.M.A. Stuart, A.E.Tzavaras}, Weak-strong uniqueness of dissipative measure-valued solutions for
polyconvex elastodynamics, submitted (2011).

\bibitem{DiPerna79}
{\sc R. DiPerna}, Uniqueness of solutions to hyperbolic conservation laws, {\it Indiana U. Math. J.} {\bf 28} (1979),
137-188.

\bibitem{DiPerna83}
{\sc R. DiPerna}, Convergence of approximate solutions to conservation laws, {\it Arch. Rational Mech. Analysis} {\bf
82} (1983), 27-70.


\bibitem{LT}
{\sc  C. Lattanzio, A.E. Tzavaras}, Structural properties of stress relaxation and convergence from viscoelasticity to
polyconvex elastodynamics, {\itshape Arch. Rational Mech. Anal.} {\bfseries 180} (2006), 449--492.


\bibitem{MT12}
A. Miroshnikov, A. Tzavaras, {\it Convergence of variational approximation schemes for elastodynamics with polyconvex
energy}, Journ. Anal. Appl. (2014).

\bibitem{Qin98}
{\sc T. Qin}, Symmetrizing nonlinear elastodynamic system, {\itshape J. Elasticity} {\bfseries 50} (1998), 245-252.




\bibitem{Tr}
{\sc C. Truesdell, W. Noll}, {\it The non-linear field theories of mechanics}, Handbuch der Physik {\bfseries III}, 3
(Ed. S.Fl\"{u}gge), Berlin: Springer (1965)

\end{thebibliography}
